\documentclass[11pt,a4paper]{amsart}
\usepackage[utf8]{inputenc}
\usepackage[english]{babel}
\usepackage{amsmath}
\usepackage{amsfonts}
\usepackage{amssymb}
\usepackage{mathabx}
\usepackage{graphicx}
\usepackage{pstricks}
\usepackage{float}
\usepackage{fullpage}
\usepackage{stmaryrd}
\usepackage{hyperref}

\newcommand{\R}{\mathbb R}

\newcommand{\E}{\mathbb E}

\newcommand{\N}{\mathbb N}
\newcommand{\Y}{\mathsf Y}
\newcommand{\F}{\mathcal F}
\newcommand{\D}{\mathcal D}
\newcommand{\M}{\mathcal M}
\newcommand{\I}{\mathcal I}
\newcommand{\G}{\mathcal G}
\newcommand{\cY}{\widecheck{\mathsf{Y}}}
\newcommand{\calL}{\mathcal L}
\newcommand{\1}{\mathbf 1}
\newcommand{\Perm}{\operatorname{Perm}}
\newcommand{\Tr}{\operatorname{Tr}}
\newcommand{\Ren}{\operatorname{Re}}

\newtheorem{thm}{Theorem}[section]
\newtheorem{lemma}[thm]{Lemma}
\newtheorem{defn}[thm]{Definition}
\newtheorem{prop}[thm]{Proposition}

\newtheorem{cor}[thm]{Corollary}

\newtheorem{rmk}[thm]{Remark}
\newtheorem{quest}[thm]{Question}

\theoremstyle{remark}

\numberwithin{equation}{section}

\author{Titus Lupu}
\address {CNRS and LPSM, UMR 8001,
Sorbonne Université,
4 place Jussieu,
75252 Paris cedex 05,
France}
\email
{titus.lupu@upmc.fr}

\title{Isomorphisms of $\beta$-Dyson's Brownian motion with Brownian local time}

\begin{document}

\begin{abstract}
We show that the Brydges-Fröhlich-Spencer-Dynkin and the Le Jan's isomorphisms between the Gaussian free fields and the occupation times of symmetric Markov processes generalize to 
the $\beta$-Dyson's Brownian motion.
For $\beta\in\{1,2,4\}$ this is a consequence of the Gaussian case, however the relation holds for general $\beta$.
We further raise the question whether there is an analogue of 
$\beta$-Dyson's Brownian motion on general electrical networks, interpolating and extrapolating the fields of eigenvalues in 
matrix-valued Gaussian free fields. 
In the case $n=2$ we give a simple construction.
\end{abstract}


\maketitle

\section{Introduction}

There is a class of results, known as \textit{isomorphism theorems},
relating the squares of Gaussian free fields (GFFs) to occupation times of symmetric Markov processes.
They originate from the works in mathematical physics
\cite{Symanzik1969QFT,BFS82Loop}. 
For a review, see \cite{MarcusRosen2006MarkovGaussianLocTime,Sznitman2012LectureIso}.
Here in particular we will be interested in
the Brydges-Fröhlich-Spencer-Dynkin isomorphism
\cite{BFS82Loop,Dynkin1984Isomorphism,Dynkin1984IsomorphismPresentation}
and in the Le Jan's  isomorphism
\cite{LeJan2010LoopsRenorm,LeJan2011Loops}.
The BFS-Dynkin isomorphism involves Markovian paths with fixed ends.
Le Jan's  isomorphism involves a Poisson point process of Markovian loops, with an intensity parameter $\alpha=1/2$ in the case of real scalar GFFs. For vector-valued GFFs with $d$ components,
the intensity parameter is $\alpha= d/2$.
We show that both Le Jan's and BFS-Dynkin isomorphisms have
a generalization to $\beta$-Dyson's Brownian motion,
and provide identities relating the latter to local times of
one-dimensional Brownian motions.
By doing so, we go beyond the Gaussian setting.

For $\beta\in\{1,2,4\}$, a $\beta$-Dyson's Brownian motion
is the diffusion of eigenvalues in a Brownian motion on the space of
real symmetric $(\beta=1)$, complex Hermitian $(\beta=2)$, respectively quaternionic Hermitian $(\beta=4)$ matrices. 
Yet, the $\beta$-Dyson's Brownian motion is defined for every $\beta\geq 0$.
The one-dimensional marginals of $\beta$-Dyson's Brownian motion are
Gaussian beta ensembles G$\beta$E.
The generalization of Le Jan's and BFS-Dynkin isomorphisms
works for every $\beta\geq 0$,
and for $\beta\in\{1,2,4\}$ it follows from the Gaussian case.
The intensity parameter $\alpha$ appearing in the Le Jan's
type isomorphism is given by
\begin{displaymath}
2\alpha =d(\beta,n)= n +n(n-1)\dfrac{\beta}{2},
\end{displaymath}
where $n$ is the number of "eigenvalues".
In particular, $\alpha$ takes not only half-integer values, as in the Gaussian case, but a whole half-line of values.
The BFS-Dynkin type isomorphism involves polynomials defined by a recurrence with a structure similar to that of the
Schwinger-Dyson equation for G$\beta$E.
These polynomials also give the symmetric moments of the 
$\beta$-Dyson's Brownian motion.

We further ask the question whether an analogue of 
G$\beta$E and $\beta$-Dyson's Brownian motion
could exist on electrical networks and interpolate and extrapolate the distributions of the eigenvalues in matrix-valued GFFs.
Our motivation for this is that such analogues could be related to
Poisson point process of random walk loops, 
in particular to those of non half-integer intensity parameter.
If the underlying graph is a tree, the construction of such analogues is straightforward, by taking $\beta$-Dyson's Brownian motions along each branch of the tree. 
However, if the graph contains cycles, this is not immediate, and one does not expect a Markov property for the obtained fields.
However, in the simplest case $n=2$, 
we provide a construction working on any graph.

Our article is organized as follows.
In Section \ref{Sec Iso 1D} we recall the BFS-Dynkin and the Le Jan's
isomorphisms in the particular case of 1D Brownian motion.
In Section \ref{Sec GbE} we recall the definition of
Gaussian beta ensembles and the corresponding Schwinger-Dyson equation.
Section \ref{Sec b Dyson} deals with
$\beta$-Dyson's Brownian motion and the corresponding 
isomorphisms.
Section \ref{Sec Graph} deals with general electrical networks.
We give our construction for $n=2$
and ask our questions for $n\geq 3$.

\section{Isomorphism theorems for 1D Brownian motion}
\label{Sec Iso 1D}

Let $(B_{t})_{t\geq 0}$ be the standard Brownian motion on $\R$.
$L^{x}$ will denote the Brownian local times:
\begin{displaymath}
L^{x}((B_{s})_{0\leq s\leq t})
=\lim_{\varepsilon\to 0}
\dfrac{1}{2\epsilon}
\int_{0}^{t}\1_{\vert B_{s}-x\vert<\varepsilon} ds.
\end{displaymath}
We will denote by $p(t,x,y)$ the heat kernel on $\R$,
and by $p_{\R_{+}}(t,x,y)$ the heat kernel on $\R_{+}$
with condition $0$ in $0$:
\begin{displaymath}
p(t,x,y)=\dfrac{1}{\sqrt{2\pi t}}
e^{-\frac{(y-x)^{2}}{2t}},
\qquad
p_{\R_{+}}(t,x,y)
= p(t,x,y) - p(t,x,-y).
\end{displaymath}
We will denote by $\mathbb{P}^{t,x,y}(\cdot)$
the Brownian bridge probability from $x$ to $y$ in 
time $t$, 
and by $\mathbb{P}_{\R_{+}}^{t,x,y}(\cdot)$ (for $x,y>0$) the
probability measures where one conditions $\mathbb{P}^{t,x,y}(\cdot)$
on that the bridge does not hit $0$.
Let $(G_{\R_{+}}(x,y))_{x,y\geq 0}$ be the Green's function
of $\frac{1}{2}\frac{d^{2}}{dx^{2}}$ on
$\R_{+}$ with $0$ condition in $0$,
and for $K>0$, 
$(G_{K}(x,y))_{x,y\geq 0}$ the Green's function
of $\frac{1}{2}\frac{d^{2}}{dx^{2}}-K$ on $\R$:
\begin{eqnarray*}
G_{\R_{+}}(x,y) &=& 2 x\wedge y
= \int_{0}^{+\infty} p_{\R_{+}}(t,x,y) dt,
\\
G_{K}(x,y) &=& \dfrac{1}{\sqrt{2 K}} e^{-\sqrt{2 K}\vert y-x\vert}
= \int_{0}^{+\infty} p(t,x,y)e^{-Kt} dt.
\end{eqnarray*}

Let $(\mu_{\R_{+}}^{x,y})_{x,y > 0}$,
resp. $(\mu_{K}^{x,y})_{x,y\in\R}$
be the following measures on finite-duration paths:
\begin{equation}
\label{Eq paths 1D}
\mu_{\R_{+}}^{x,y}(\cdot): = 
\int_{0}^{+\infty}\mathbb{P}_{\R_{+}}^{t,x,y}(\cdot)
p_{\R_{+}}(t,x,y) dt,
\qquad
\mu_{K}^{x,y}(\cdot): = 
\int_{0}^{+\infty}\mathbb{P}^{t,x,y}(\cdot)
p(t,x,y) e^{-Kt} dt.
\end{equation}
The total mass of $\mu_{\R_{+}}^{x,y}$,
resp. $\mu_{K}^{x,y}$, 
is $G_{\R_{+}}(x,y)$, resp. $G_{K}(x,y)$.
The image of $\mu_{\R_{+}}^{x,y}$, resp. $\mu_{K}^{x,y}$,
by time reversal is 
$\mu_{\R_{+}}^{y,x}$, resp. $\mu_{K}^{y,x}$.

Let $T_{x}$ denote the first hitting time of a level $x$ by the Brownian motion $(B_{t})_{t\geq 0}$. 
We will denote by $\gamma$ a generic path on $\R$.
Let $(\check{\mu}^{x,y}(\cdot))_{x< y\in\R}$,
resp. $(\check{\mu}_{K}^{x,y}(\cdot))_{x< y\in\R}$
be the following measures on paths from $x$ to $y$:
\begin{displaymath}
\check{\mu}^{x,y}(F(\gamma))
=\mathbb{E}_{B_{0}=y}[F((B_{T_{x}-t})_{0\leq t\leq T_{x}})],
\qquad
\check{\mu}_{K}^{x,y}(F(\gamma))
=\mathbb{E}_{B_{0}=y}\big[e^{-K T_{x}}
F((B_{T_{x}-t})_{0\leq t\leq T_{x}})\big].
\end{displaymath}
The measure $\check{\mu}^{x,y}$ has total mass $1$ (probability measure),
whereas the total mass of 
$\check{\mu}_{K}^{x,y}$ is
\begin{displaymath}
\mathbb{E}_{B_{0}=y}\big[e^{-K T_{x}}\big] = 
e^{-\sqrt{2 K}\vert y-x\vert}
=\dfrac{G_{K}(x,y)}{G_{K}(x,x)}.
\end{displaymath}
For $0<x\leq y<z$, the measure $\mu_{\R_{+}}^{x,z}$
can be obtained as the image of the product measure
$\mu_{\R_{+}}^{x,y}\otimes\check{\mu}^{y,z}$
under the concatenation of two paths.
Similarly, for $x\leq y<z\in\R$,
the measure $\mu_{K}^{x,z}$ is the image of
$\mu_{K}^{x,y}\otimes\check{\mu}_{K}^{y,z}$
under the concatenation of two paths.

Let $(W(x))_{x\in R}$ denote a two-sided Brownian motion, 
i.e. $(W(x))_{x\geq 0}$ and $(W(-x))_{x\geq 0}$
being two independent standard Brownian motions starting from $0$
($W(0)=0$).
Note that here $x$ is rather a one-dimensional space variable
then a time variable.
The derivative $dW(x)$ is a white noise on $\R$.
Let $(\phi_{\R_{+}}(x))_{x\geq 0}$ denote the process
$(\sqrt{2}W(x))_{x\geq 0}$. 
The covariance function of $\phi_{\R_{+}}$ is $G_{\R_{+}}$.
Let $(\phi_{K}(x))_{x\in \R}$ be the stationary
Ornstein–Uhlenbeck process with invariant measure
$\mathcal{N}(0,1/\sqrt{2K})$. 
It is a solution to the SDE
\begin{displaymath}
d\phi_{K}(x) = \sqrt{2} dW(x) - \sqrt{2K} \phi_{K}(x) dx.
\end{displaymath}
The covariance function of $\phi_{K}$ is $G_{K}$.

What follows is the BFS-Dynkin isomorphism
(Theorem 2.2 in \cite{BFS82Loop}, 
Theorems 6.1 and 6.2 in \cite{Dynkin1984Isomorphism},
Theorem 1 in \cite{Dynkin1984IsomorphismPresentation})
in the particular case of a 1D Brownian motion.
In general, the BFS-Dynkin isomorphism relates the
squares of Gaussian free fields to local times of
symmetric Markov processes.

\begin{thm}[Brydges-Fröhlich-Spencer \cite{BFS82Loop},
Dynkin
\cite{Dynkin1984Isomorphism,Dynkin1984IsomorphismPresentation}]
\label{ThmIsoDynkin}
Let $F$ be a bounded measurable functional on
$\mathcal{C}(\R_{+})$, resp. on $\mathcal{C}(\R)$.
Let $k\geq 1$ and
$x_{1},x_{2},\dots,x_{2k}$
in $(0,+\infty)$, resp. in $\mathbb{R}$. 
Then
\begin{displaymath}
\mathbb{E}\Big[
\prod_{i=1}^{2k}\phi_{\R_{+}}(x_{i})
F(\phi_{\R_{+}}^{2}/2)
\Big]=
\sum_{\substack{(\{a_{i},b_{i}\})_{1\leq i\leq k}\\
\text{partition in pairs}
\\
\text{of } \{1,2,\dots, 2k\}
}} 
\int\displaylimits_{\gamma_{1},\dots ,\gamma_{k}}
\E\Big[
F(\phi_{\R_{+}}^{2}/2+L(\gamma_{1})+\dots + L(\gamma_{k}))
\Big]
\prod_{i=1}^{k}
\mu^{x_{a_{i}},x_{b_{i}}}_{\R_{+}}
(d\gamma_{i}),
\end{displaymath}
resp.
\begin{displaymath}
\mathbb{E}\Big[
\prod_{i=1}^{2k}\phi_{K}(x_{i})
F(\phi_{K}^{2}/2)
\Big]=
\sum_{\substack{(\{a_{i},b_{i}\})_{1\leq i\leq k}\\
\text{partition in pairs}
\\
\text{of } \{1,2,\dots, 2k\}
}} 
\int\displaylimits_{\gamma_{1},\dots ,\gamma_{k}}
\E\Big[
F(\phi_{K}^{2}/2+L(\gamma_{1})+\dots + L(\gamma_{k}))
\Big]
\prod_{i=1}^{k}
\mu^{x_{a_{i}},x_{b_{i}}}_{K}
(d\gamma_{i}),
\end{displaymath}
where the sum runs over the
$(2k)!/(2^{k}k!)$ partitions in pairs,
the $\gamma_{i}$-s are Brownian paths
and the $L(\gamma_{i})$-s are the corresponding occupation fields
$x\mapsto L^{x}(\gamma_{i})$.
\end{thm}

\begin{rmk}
\label{Rmk chech}
Since for $x<y$, the measure
$\mu_{\R_{+}}^{x,y}$, resp. $\mu_{K}^{x,y}$,
can be decomposed as
$\mu_{\R_{+}}^{x,x}\otimes\check{\mu}^{x,y}$,
resp.
$\mu_{K}^{x,x}\otimes\check{\mu}_{K}^{x,y}$,
Theorem \ref{ThmIsoDynkin} can be rewritten using only the 
measures of type $\mu_{\R_{+}}^{x,x}$ and $\check{\mu}^{x,y}$,
resp.
$\mu_{K}^{x,x}$ and $\check{\mu}_{K}^{x,y}$.
\end{rmk}

\bigskip

To a wide class of symmetric Markov processes one can associate in a natural way an infinite, $\sigma$-finite measure on loops
\cite{LawlerWerner04BMLoopSoup,LawlerFerreras07RWLoopSoup,
LawlerLimic2010RW,LeJan2010LoopsRenorm,LeJan2011Loops,
LeJanMarcusRosen12Loops,FitzsimmonsRosen14LoopSoup}.
It originated from the works in mathematical physics
\cite{Symanzik65Scalar,Symanzik66Scalar,Symanzik1969QFT,BFS82Loop}.
Here we recall it in the setting of a 1D Brownian motion,
which has been studied in \cite{Lupu20131dimLoops}.
The range of a loop will be just a segment on the line,
but it will carry a non-trivial Brownian local time process
which will be of interest for us.

Given a Brownian loop $\gamma$, $T(\gamma)$ will denote its duration.
The measures on (rooted) loops are
\begin{equation}
\label{Eq loop 1D}
\mu_{\R_{+}}^{\rm loop}(d\gamma):=
\dfrac{1}{T(\gamma)}\int_{\R_{+}}
\mu^{x,x}_{\R_{+}}(d\gamma) dx,
\qquad
\mu_{K}^{\rm loop}(d\gamma)=
\dfrac{1}{T(\gamma)}\int_{\R}
\mu^{x,x}_{K}(d\gamma) dx.
\end{equation}
Usually one considers unrooted loops, 
but this will not be important here.
The 1D Brownian \textit{loop soups} are the Poisson point processes,
denoted $\calL^{\alpha}_{\R_{+}}$, 
resp. $\calL^{\alpha}_{K}$,
of intensity $\alpha\mu_{\R_{+}}^{\rm loop}$,
resp. $\alpha\mu_{K}^{\rm loop}$,
where $\alpha>0$ is an intensity parameter.
$L(\calL^{\alpha}_{\R_{+}})$, resp.
$L(\calL^{\alpha}_{K})$, will denote the occupation field of
$\calL^{\alpha}_{\R_{+}}$, resp. $\calL^{\alpha}_{K}$:
\begin{displaymath}
L^{x}(\calL^{\alpha}_{\R_{+}}):=
\sum_{\gamma\in\calL^{\alpha}_{\R_{+}}}
L^{x}(\gamma),
\qquad
L^{x}(\calL^{\alpha}_{K}):=
\sum_{\gamma\in\calL^{\alpha}_{K}}
L^{x}(\gamma).
\end{displaymath}

The following statement deals with the law of
$L(\calL^{\alpha}_{\R_{+}})$, resp. $L(\calL^{\alpha}_{K})$.
See Proposition 4.6, Property 4.11 and Corollary 5.5 in \cite{Lupu20131dimLoops}.
For the analogous statements in discrete space setting,
see Corollary 5, Proposition 6, Theorem 13 in \cite{LeJan2010LoopsRenorm} 
and Corollary 1, Section 4.1, Proposition 16, Section 4.2,
Theorem 2, Section 5.1 in \cite{LeJan2011Loops}.
In general, one gets $\alpha$-permanental fields
(see also \cite{LeJanMarcusRosen12Loops,FitzsimmonsRosen14LoopSoup}).
For $\alpha=\frac{1}{2}$ in particular, one gets square Gaussians.
We recall that given a matrix
$M=(M_{ij})_{1\leq i,j\leq k}$, its $\alpha$-permanent is
\begin{equation}
\label{Eq Perm}
\Perm_{\alpha}(M):=
\sum_{\substack{\sigma \text{ permutation}
\\ \text{of } \{1,2,\dots ,k\}}}
\alpha^{\# \text{ cycles of } \sigma}
\prod_{i=1}^{k} M_{i \sigma(i)}.
\end{equation}

\begin{thm}[Le Jan \cite{LeJan2010LoopsRenorm,LeJan2011Loops},
Lupu \cite{Lupu20131dimLoops}]
\label{Thm loop soup}
For every $\alpha>0$ and
$x\in\R_{+}$, resp. $x\in \R$, the r.v.
$L^{x}(\calL^{\alpha}_{\R_{+}})$, resp. $L^{x}(\calL^{\alpha}_{K})$,
follows the distribution
$\operatorname{Gamma}(\alpha,G_{\R_{+}}(x,x)^{-1})$,
resp.
$\operatorname{Gamma}(\alpha,G_{K}(x,x)^{-1})$.
Moreover, the process
$\alpha\mapsto L^{x}(\calL^{\alpha}_{\R_{+}})$,
resp. $L^{x}(\calL^{\alpha}_{K})$, is a pure jump Gamma subordinator
with Lévy measure
\begin{displaymath}
\1_{l>0} \dfrac{e^{-l/G_{\R_{+}}(x,x)}}{l} dl,
\qquad
\text{resp. }
\1_{l>0} \dfrac{e^{-l/G_{K}(x,x)}}{l} dl.
\end{displaymath}
Let $x_{1},x_{2},\dots, x_{k}\in \R_{+}$,
resp. $\R$. 
Then
\begin{displaymath}
\E\Big[\prod_{i=1}^{k}L^{x_{i}}(\calL^{\alpha}_{\R_{+}})\Big] =
\Perm_{\alpha}\big(G_{\R_{+}}(x_{i},x_{j})_{1\leq i,j\leq k}\big),
~~
\E\Big[\prod_{i=1}^{k}L^{x_{i}}(\calL^{\alpha}_{K})\Big] =
\Perm_{\alpha}\big(G_{K}(x_{i},x_{j})_{1\leq i,j\leq k}\big)
.
\end{displaymath}
For $x\geq 0$, $x\mapsto L^{x}(\calL^{\alpha}_{\R_{+}})$ is a
solution to the SDE
\begin{displaymath}
dL^{x}(\calL^{\alpha}_{\R_{+}}) = 
2\big(L^{x}(\calL^{\alpha}_{\R_{+}})\big)^{\frac{1}{2}}
dW(x)+ 2\alpha dx,
\end{displaymath}
with initial condition $L^{0}(\calL^{\alpha}_{\R_{+}})=0$.
That is to say it is a square Bessel process of dimension
$2\alpha$,
reflected at level $0$ for $\alpha <1$.
For $x\in\mathbb{R}$, 
$x\mapsto L^{x}(\calL^{\alpha}_{K})$ is a stationary solution to the
SDE
\begin{displaymath}
dL^{x}(\calL^{\alpha}_{K}) = 
2\big(L^{x}(\calL^{\alpha}_{K})\big)^{\frac{1}{2}}
dW(x) -2\sqrt{2K}L^{x}(\calL^{\alpha}_{K}) + 2\alpha dx.
\end{displaymath}
In particular, for $\alpha=\frac{1}{2}$,
one has the following identities in law between stochastic processes:
\begin{equation}
\label{Iso Le Jan 1D}
L(\calL^{\alpha}_{\R_{+}})\stackrel{(\text{ law })}{=}
\dfrac{1}{2}\phi_{\R_{+}}^{2},
\qquad
L(\calL^{\alpha}_{K})\stackrel{(\text{ law })}{=}
\dfrac{1}{2}\phi_{K}^{2}.
\end{equation}
\end{thm}

\section{Gaussian beta ensembles}
\label{Sec GbE}

For references on Gaussian beta ensembles, see
\cite{DumitriuEdelman02TriDiagBeta,OxfordHandbookRandomMatrixBeta},
\cite[Section~1.2.2]{EynardKimuraRibault18RandomMatrices},
and 
\cite[Section~4.5]{AndersonGuionnetZeitouni09RandomMatrix}.
Fix $n\geq 2$. 
For 
$\lambda=(\lambda_{1},\lambda_{2},\dots ,\lambda_{n})\in\R^{n}$,
$D(\lambda)$ will denote the Vandermonde determinant
\begin{displaymath}
D(\lambda):=
\prod_{1\leq j<j'\leq n}(\lambda_{j'}-\lambda_{j}).
\end{displaymath}
For $q\geq 1$, $p_{q}(\lambda)$ will denote the $q$-th power sum polynomial
\begin{displaymath}
p_{q}(\lambda):=\sum_{j=1}^{n}\lambda_{j}^{q} .
\end{displaymath}
By convention,
\begin{displaymath}
p_{0}(\lambda)=n.
\end{displaymath}
A Gaussian beta ensemble G$\beta$E, with $\beta> -\frac{2}{n}$, follows the distribution
\begin{equation}
\label{Eq GbE}
\dfrac{1}{Z_{\beta,n}}
\vert D(\lambda)\vert^{\beta}
e^{-\frac{1}{2}p_{2}(\lambda)} 
\prod_{j=1}^{n}d\lambda_{j},
\end{equation}
where $Z_{\beta,n}$ is given by 
(\cite[Formula~(17.6.7)]{Mehta04RandomMatrices}
 and
\cite[Formula~(1.2.23)]{EynardKimuraRibault18RandomMatrices})
\begin{displaymath}
Z_{\beta,n}=(2\pi)^{\frac{n}{2}}
\prod_{j=1}^{n}
\dfrac{\Gamma\big(1+j\frac{\beta}{2}\big)}
{\Gamma\big(1+\frac{\beta}{2}\big)}.
\end{displaymath}
The brackets $\langle\cdot\rangle_{\beta,n}$
will denote the expectation with respect to
\eqref{Eq GbE}.
For $\beta=0$ one gets $n$ i.i.d. $\mathcal{N}(0,1)$ Gaussians.
For $\beta$ equal to 1, 2, resp. 4, one gets the eigenvalue distribution of GOE, GUE, resp. GSE random matrices
\cite{Mehta04RandomMatrices,EynardKimuraRibault18RandomMatrices}.
Usually the G$\beta$E are studied for $\beta>0$
\cite{DumitriuEdelman02TriDiagBeta}, 
but the distribution \eqref{Eq GbE} is well defined for all
$\beta>-\frac{2}{n}$. 
For $\beta\in (-\frac{2}{n},0)$ there is an attraction between
the $\lambda_{j}$-s instead of a repulsion as for
$\beta >0$.
Moreover, as $\beta\to -\frac{2}{n}$, 
$\lambda$ under \eqref{Eq GbE} converges in law to
\begin{equation}
\label{Eq n ident Gauss}
\Big(\dfrac{1}{\sqrt{n}}\xi,\dfrac{1}{\sqrt{n}}\xi,
\dots ,\dfrac{1}{\sqrt{n}}\xi\Big),
\end{equation}
where $\xi$ follows $\mathcal{N}(0,1)$.

Let $d(\beta,n)$ denote
\begin{displaymath}
d(\beta,n) = n + n(n-1)\dfrac{\beta}{2}.
\end{displaymath}
One can see $d(\beta,n)$ as a kind of pseudo-dimension.
For $\beta\in\{1,2,4\}$, $d(\beta,n)$ is the dimension of the corresponding space of matrices.

Let $\nu=(\nu_{1},\nu_{2},\dots,\nu_{m})$, where
$m\geq 1$, and for all
$k\in\{1,2,\dots,m\}$,
$\nu_{k}\in\mathbb{N}\setminus\{0\}$.
We will denote
\begin{displaymath}
m(\nu) = m, \qquad
\vert\nu\vert = \sum_{k=1}^{m(\nu)}\nu_{k}.
\end{displaymath}
Let $p_{\nu}(\lambda)$ denote
\begin{displaymath}
p_{\nu}(\lambda):=
\prod_{k=1}^{m(\nu)}
p_{\nu_{k}}(\lambda).
\end{displaymath}
By convention, we set
$p_{\emptyset}(\lambda)=1$ and
$\vert\emptyset\vert=0$.
Note that $p_{\emptyset}(\lambda)\neq p_{0}(\lambda)$.
We are interested in the expression of the moments
$\langle p_{\nu}(\lambda)\rangle_{\beta,n}$.
These are 0 if $\vert\nu\vert$ is not even.
For $\vert\nu\vert$ even, these moments are given by a recurrence
known as \textit{loop equation or
Schwinger-Dyson equation}
(\cite[Lemma~4.13]{LaCroix09PhD},
\cite[slide~$3/15$]{LaCroix13SlidesBeta}
and \cite[Section~4.1.1]{EynardKimuraRibault18RandomMatrices}).
See the Appendix for the expressions of some moments.

\begin{prop}[Schwinger-Dyson equation
\cite{LaCroix09PhD,LaCroix13SlidesBeta,
EynardKimuraRibault18RandomMatrices}]
\label{Prop SD beta}
For every $\beta> -2/n$ and 
every $\nu$ as above with $\vert\nu\vert$ even,
\begin{eqnarray}
\label{Eq SD beta}
\langle p_{\nu}(\lambda)\rangle_{\beta,n} &=& 
\dfrac{\beta}{2}\sum_{i=1}^{\nu_{m(\nu)}-1}
\langle p_{(\nu_{r})_{r\neq m(\nu)}}(\lambda)
p_{i-1}(\lambda)
p_{\nu_{m(\nu)}-1-i}(\lambda)\rangle_{\beta,n}
\\ 
\nonumber
&& +
\Big(1-\dfrac{\beta}{2}\Big)(\nu_{m(\nu)}-1)
\langle p_{(\nu_{r})_{r\neq m(\nu)}}(\lambda)
p_{\nu_{m(\nu)}-2}(\lambda)\rangle_{\beta,n}
\\ 
\nonumber
&& +
\sum_{k=1}^{m(\nu)-1} \nu_{k}
\langle p_{(\nu_{r})_{r\neq k,m(\nu)}}(\lambda)
p_{\nu_{k}+\nu_{m(\nu)}-2}(\lambda)\rangle_{\beta,n}
,
\end{eqnarray}
where $p_{0}(\lambda)=n$.
In particular, for $q$ even,
\begin{displaymath}
\langle p_{q}(\lambda)\rangle_{\beta,n} =
\dfrac{\beta}{2}\sum_{i=1}^{q-1}
\langle 
p_{i-1}(\lambda)
p_{q-1-i}(\lambda)\rangle_{\beta,n}
+
\Big(1-\dfrac{\beta}{2}\Big)(q-1)
\langle p_{q-2}(\lambda)\rangle_{\beta,n}
,
\end{displaymath}
and for $\nu$ with $ \nu_{m(\nu)}=1$,
\begin{displaymath}
\langle p_{\nu}(\lambda)\rangle_{\beta,n} =
\sum_{k=1}^{m(\nu)-1} \nu_{k}
\langle p_{(\nu_{r})_{r\neq k,m(\nu)}}(\lambda)
p_{\nu_{k}-1}(\lambda)\rangle_{\beta,n}
.
\end{displaymath}
The recurrence \eqref{Eq SD beta} and the initial condition
$p_{0}(\lambda)=n$ determine all the moments
$\langle p_{\nu}(\lambda)\rangle_{\beta,n}$.
\end{prop}

\begin{proof}
Note that \eqref{Eq SD beta} determines the moments 
$\langle p_{\nu}(\lambda)\rangle_{\beta,n}$ because on the left-hand side one has a degree $\vert\nu\vert$, 
and on the right-hand side all the terms have a degree
$\vert\nu\vert-2$.
It is enough to check \eqref{Eq SD beta}
for $\beta>0$, since both sides are analytic in $\beta$.
For $\beta>0$,
we outline the proof appearing in
\cite[Lemma~4.13]{LaCroix09PhD} and
\cite[Section~4.1.1]{EynardKimuraRibault18RandomMatrices},
so as to be self-contained.
Let us denote here
$\tilde{\nu}:=(\nu_{1},\nu_{2},\dots,\nu_{m(\nu)-1})$,
so that
$p_{\nu}(\lambda)= p_{\nu_{m(\nu)}}(\lambda)p_{\tilde{\nu}}(\lambda)$.
We have that
\begin{eqnarray*}
\dfrac{\partial}{\partial \lambda_{1}}
\Big(
\lambda_{1}^{\nu_{m(\nu)}-1}
p_{\tilde{\nu}}(\lambda)
\vert D(\lambda)\vert^{\beta}
e^{-\frac{1}{2}p_{2}(\lambda)}
\Big)
&=&
-\lambda_{1}^{\nu_{m(\nu)}}
p_{\tilde{\nu}}(\lambda)
\vert D(\lambda)\vert^{\beta}
e^{-\frac{1}{2}p_{2}(\lambda)}
\\
&&+ \beta \sum_{j=2}^{n}
\dfrac{\lambda_{1}^{\nu_{m(\nu)}-1}}{\lambda_{1}-\lambda_{j}}
p_{\tilde{\nu}}(\lambda)
\vert D(\lambda)\vert^{\beta}
e^{-\frac{1}{2}p_{2}(\lambda)}
\\
&&+ (\nu_{m(\nu)}-1)
\lambda_{1}^{\nu_{m(\nu)}-2}
p_{\tilde{\nu}}(\lambda)
\vert D(\lambda)\vert^{\beta}
e^{-\frac{1}{2}p_{2}(\lambda)}
\\
&&+
\sum_{k=1}^{m(\nu)-1}
\nu_{k}
\lambda_{1}^{\nu_{k}+\nu_{m(\nu)}-2}
p_{(\nu_{r})_{r\neq k,m(\nu)}}(\lambda)
\vert D(\lambda)\vert^{\beta}
e^{-\frac{1}{2}p_{2}(\lambda)}.
\end{eqnarray*}
Since
\begin{displaymath}
\int_{\R}
\dfrac{\partial}{\partial \lambda_{1}}
\Big(
\lambda_{1}^{\nu_{m(\nu)}-1}
p_{\tilde{\nu}}(\lambda)
\vert D(\lambda)\vert^{\beta}
e^{-\frac{1}{2}p_{2}(\lambda)}
\Big)
d\lambda_{1} = 0,
\end{displaymath}
we get that
\begin{eqnarray*}
\langle
\lambda_{1}^{\nu_{m(\nu)}}
p_{\tilde{\nu}}(\lambda)
\rangle_{\beta,n}
&=&
\beta
\sum_{j=2}^{n}
\bigg\langle
\dfrac{\lambda_{1}^{\nu_{m(\nu)}-1}}{\lambda_{1}-\lambda_{j}}
p_{\tilde{\nu}}(\lambda)
\bigg\rangle_{\beta,n}
+
(\nu_{m(\nu)}-1)
\langle
\lambda_{1}^{\nu_{m(\nu)}-2}
p_{\tilde{\nu}}(\lambda)
\rangle_{\beta,n}
\\&&+
\sum_{k=1}^{m(\nu)-1}
\nu_{k}
\langle
\lambda_{1}^{\nu_{k}+\nu_{m(\nu)}-2}
p_{(\nu_{r})_{r\neq k,m(\nu)}}(\lambda)
\rangle_{\beta,n}
.
\end{eqnarray*}
Analogous relations hold for all other indices
$j'\in\{2,\dots,n\}$.
By summing over $j'\in\{1,2,\dots,n\}$,
we get
\begin{eqnarray*}
\langle
p_{\nu}(\lambda)
\rangle_{\beta,n}
&=&
\beta
\sum_{1\leq j <j'\leq n}
\bigg\langle
\dfrac{\lambda_{j}^{\nu_{m(\nu)}-1}-\lambda_{j'}^{\nu_{m(\nu)}-1}}
{\lambda_{j}-\lambda_{j'}}
p_{\tilde{\nu}}(\lambda)
\bigg\rangle_{\beta,n}
+
(\nu_{m(\nu)}-1)
\langle
p_{\nu_{m(\nu)}-2}(\lambda)
p_{\tilde{\nu}}(\lambda)
\rangle_{\beta,n}
\\&&+
\sum_{k=1}^{m(\nu)-1}
\nu_{k}
\langle
p_{\nu_{k}+\nu_{m(\nu)}-2}(\lambda)
p_{(\nu_{r})_{r\neq k,m(\nu)}}(\lambda)
\rangle_{\beta,n}
.
\end{eqnarray*}
Furthermore,
\begin{displaymath}
\sum_{1\leq j <j'\leq n}
\dfrac{\lambda_{j}^{\nu_{m(\nu)}-1}-\lambda_{j'}^{\nu_{m(\nu)}-1}}
{\lambda_{j}-\lambda_{j'}}
=
-\dfrac{1}{2}(\nu_{m(\nu)}-1)p_{\nu_{m(\nu)}-2}(\lambda)
+\dfrac{1}{2}
\sum_{i=1}^{\nu_{m(\nu)-1}}
p_{i-1}(\lambda)
p_{\nu_{m(\nu)}-1-i}(\lambda).
\end{displaymath}
So we get \eqref{Eq SD beta}.
\end{proof}

Next are some elementary properties of G$\beta$E,
which follow from the form of the density \eqref{Eq GbE}.

\begin{prop}
\label{Proberty GbE}
The following holds.
\begin{enumerate}
\item For every $\beta > -2/n$, $\frac{1}{\sqrt{n}}p_{1}(\lambda)$
under G$\beta$E has for distribution $\mathcal{N}(0,1)$.
\item For every $\beta > -2/n$, $p_{2}(\lambda)/2$
under G$\beta$E has for distribution
$\operatorname{Gamma}(d(\beta,n)/2,1)$.
\item $p_{1}(\lambda)$ and 
$\lambda-\frac{1}{n}p_{1}(\lambda)$
under G$\beta$E are independent.
\item $\frac{1}{2}
\big(p_{2}(\lambda)-\frac{1}{n}p_{1}(\lambda)^{2}\big)
=\frac{1}{2}p_{2}\big(\lambda-\frac{1}{n}p_{1}(\lambda)\big)$
under G$\beta$E has for distribution
$\operatorname{Gamma}((d(\beta,n)-1)/2,1)$.
\end{enumerate}
\end{prop}

\begin{proof}
One can factorize the density \eqref{Eq GbE} as
\begin{displaymath}
\dfrac{1}{Z_{\beta,n}}
\Big\vert D\Big(\lambda-\frac{1}{n}p_{1}(\lambda)\Big)
\Big\vert^{\beta}
e^{-\frac{1}{2}p_{2}\big(\lambda-\frac{1}{n}p_{1}(\lambda)\big)} 
\prod_{j=1}^{n-1}
d\Big(\lambda_{j}-\frac{1}{n}p_{1}(\lambda)\Big)
\times 
e^{-\frac{1}{2n}p_{1}(\lambda)^{2}}
d p_{1}(\lambda),
\end{displaymath}
where
\begin{displaymath}
D\Big(\lambda-\frac{1}{n}p_{1}(\lambda)\Big)
= \prod_{1\leq j<j'\leq n}
\Big(\Big(\lambda_{j'}-\frac{1}{n}p_{1}(\lambda)\Big)-\Big(\lambda_{j}-\frac{1}{n}p_{1}(\lambda)\Big)\Big)
= D(\lambda).
\end{displaymath}
This immediately implies (3) and (1).
The property (2) is implied by (4), (3) and (1).
The property (4) can be obtained by computing a Laplace transform.
Fix $K>0$.
We have that
\begin{displaymath}
\big\langle e^{-\frac{1}{2}K 
p_{2}\big(\lambda-\frac{1}{n}p_{1}(\lambda)\big)}\big\rangle_{\beta,n}
=
\dfrac{1}{Z_{\beta,n}}
\int_{\R^{n}} \vert D(\lambda)\vert^{\beta}
e^{-\frac{1}{2}(K+1)p_{2}\big(\lambda-\frac{1}{n}p_{1}(\lambda)\big)
-\frac{1}{2n}p_{1}(\lambda)^{2}}
\prod_{j=1}^{n} d\lambda_{j}.
\end{displaymath}
By performing the change of variables 
$\tilde{\lambda}=(K+1)^{\frac{1}{2}}\lambda$,
we get that the expression above equals
\begin{multline*}
\dfrac{(K+1)^{-\frac{n}{2}}}{Z_{\beta,n}}
\int_{\R^{n}} 
\vert D((K+1)^{-\frac{1}{2}}\tilde{\lambda})\vert^{\beta}
e^{-\frac{1}{2}p_{2}\big(\tilde{\lambda}
-\frac{1}{n}p_{1}(\tilde{\lambda})\big)
-\frac{1}{2n(K+1)}p_{1}(\tilde{\lambda})^{2}}
\prod_{j=1}^{n} d\tilde{\lambda}_{j}
\\
=
\dfrac{(K+1)^{-\frac{1}{2} d(\beta,n)}}{Z_{\beta,n}}
\int_{\R^{n}} 
\vert D(\tilde{\lambda})\vert^{\beta}
e^{-\frac{1}{2}p_{2}(\tilde{\lambda})
+\frac{K}{2n(K+1)}p_{1}(\tilde{\lambda})^{2}}
\prod_{j=1}^{n} d\tilde{\lambda}_{j} .
\end{multline*}
Thus,
\begin{displaymath}
\big\langle e^{-\frac{1}{2}K 
p_{2}\big(\lambda-\frac{1}{n}p_{1}(\lambda)\big)}\big\rangle_{\beta,n}
=
(K+1)^{-\frac{1}{2}d(\beta,n)}
\big\langle e^{\frac{K}{2n(K+1)}p_{1}(\lambda)^{2}}
\big\rangle_{\beta,n}
=(K+1)^{-\frac{1}{2}(d(\beta,n)-1)}.
\end{displaymath}
So we get the Laplace transform of a
$\operatorname{Gamma}((d(\beta,n)-1)/2,1)$ r.v.
\end{proof}

Next is an embryonic version of the BFS-Dynkin isomorphism
(Theorem \eqref{ThmIsoDynkin}) for the G$\beta$E. 
One should imagine that the state space is reduced to one vertex, 
and a particle on it gets killed at an exponential time.

\begin{prop}
\label{Prop Dynkin GbE}
Let $\beta> -2/n$. The following holds.
\begin{enumerate}
\item Let $a\geq 0$. 
Let $h: \mathbb{R}^{n} \rightarrow \mathbb{R}$ 
be a measurable function such that
$\langle \vert h(\lambda)\vert\rangle_{\beta,n}<+\infty$.
Assume that $h$ is $a$-homogeneous, that is to say
$h(s\lambda)=s^{a}h(\lambda)$ for every $s>0$.
Let $F :[0,+\infty)\rightarrow\mathbb{R}$ be a bounded measurable function.
Let $\theta$ be a r.v. with distribution
$\operatorname{Gamma}((d(\beta,n)+a)/2,1)$.
Then
\begin{equation}
\label{Eq homog}
\langle h(\lambda)F(p_{2}(\lambda)/2)\rangle_{\beta,n}
=\langle h(\lambda)\rangle_{\beta,n}
\mathbb{E}[F(\theta)].
\end{equation}
\item In particular, let $\nu$ be a finite family of positive integers such that $\vert\nu\vert$ is even.
Let $\mathcal{T}_{1},\dots, \mathcal{T}_{\vert\nu\vert/2}$ be an
i.i.d. family of exponential times of mean $1$,
independent of the G$\beta$E. Then
\begin{displaymath}
\langle p_{\nu}(\lambda)F(p_{2}(\lambda)/2)\rangle_{\beta,n}
=\langle  p_{\nu}\rangle_{\beta,n}
\mathbb{E}\big[
\langle F(p_{2}(\lambda)/2
+\mathcal{T}_{1}+\dots +\mathcal{T}_{\vert\nu\vert/2})\rangle_{\beta,n}
\big].
\end{displaymath}
\end{enumerate}
\end{prop}

\begin{proof}
(1) clearly implies (2). It is enough to check
\eqref{Eq homog} for $F$ of form
$F(t)=e^{-K t}$, with $K>0$.
Then
\begin{eqnarray*}
\langle h(\lambda)e^{-\frac{1}{2}K p_{2}(\lambda)}\rangle_{\beta,n}&=&
\dfrac{1}{Z_{\beta,n}}
\int_{\R^{n}} h(\lambda)\vert D(\lambda)\vert^{\beta}
e^{-\frac{1}{2}(K+1)p_{2}(\lambda)}
\prod_{j=1}^{n} d\lambda_{j}
\\
&=& 
\dfrac{(K+1)^{-\frac{n}{2}}}{Z_{\beta,n}}
\int_{\R^{n}} h((K+1)^{-\frac{1}{2}}\tilde{\lambda})
\vert D((K+1)^{-\frac{1}{2}}\tilde{\lambda})\vert^{\beta}
e^{-\frac{1}{2}p_{2}(\tilde{\lambda})}
\prod_{j=1}^{n} d\tilde{\lambda}_{j}
\\
&=&
(K+1)^{-\frac{1}{2}\big(n+n(n-1)\frac{\beta}{2} + a\big)}
\langle h(\tilde{\lambda})\rangle_{\beta,n},
\end{eqnarray*}
where on the second line we used the change of variables
$\tilde{\lambda}=(K+1)^{\frac{1}{2}}\lambda$, 
and on the third line the homogeneity.
Further, 
\begin{displaymath}
(K+1)^{-\frac{1}{2}\big(n+n(n-1)\frac{\beta}{2} + a\big)}
=\E[e^{-K\theta}].
\qedhere
\end{displaymath}
\end{proof}

\section{Isomorphisms for $\beta$-Dyson's Brownian motion}
\label{Sec b Dyson}

\subsection{$\beta$-Dyson's Brownian motions and the occupation fields of 1D Brownian loop soups}

For references on $\beta$-Dyson's Brownian motion, see
\cite{Dyson62BM,Chan92DysonBM,RogersShi93DysonBM,CepaLepingle97DysonBM,
CepaLepingle07NoMultCol},
\cite[Chapter~9]{Mehta04RandomMatrices} and
\cite[Section~4.3]{AndersonGuionnetZeitouni09RandomMatrix}. 
Let $\beta\geq 0$ and $n\geq 2$. 
The $\beta$-Dyson's Brownian motion is the process 
$(\lambda(x)=
(\lambda_{1}(x),\dots,
\lambda_{n}(x)))_{x\geq 0}$
with $\lambda_{1}(x)\geq\dots\geq\lambda_{n}(x)$,
satisfying the SDE
\begin{equation}
\label{Eq b Dyson}
d\lambda_{j}(x)=
\sqrt{2} dW_{j}(x)
+ \beta \sum_{j'\neq j} 
\dfrac{dx}{\lambda_{j}(x)-\lambda_{j'}(x)},
\end{equation}
with initial condition $\lambda(0)=0$.
The derivatives $(dW_{j}(x))_{1\leq j\leq n}$
are independent white noises.
Since we will be interested in isomorphisms with Brownian local times,
the variable $x$ corresponds here to a one-dimensional spatial variable rather than a time variable.
For every $x>0$, 
$\lambda(x)/\sqrt{G_{\R_{+}}(x,x)}
=\lambda(x)/\sqrt{2x}$,
is distributed, up to a reordering of the 
$\lambda_{j}(x)$-s, as a G$\beta$E \eqref{Eq GbE}.
For $\beta$ equal to $1,2$ resp. $4$,
$(\lambda(x))_{x\geq 0}$
is the diffusion of eigenvalues
in a Brownian motion on the space of real symmetric,
complex Hermitian, resp. quaternionic Hermitian
matrices.
For $\beta\geq 1$, there is no collision between the
$\lambda_{j}(x)$-s, and for  
$\beta\in[0,1)$ two consecutive $\lambda_{j}(x)$-s
can collide, but there is no collision of three or more particles
\cite{CepaLepingle07NoMultCol}.
Note that for $\beta>0$ and $j\in\llbracket 2,n\rrbracket$,
$(\lambda_{j}(x)-\lambda_{j-1}(x))/2$
behaves near level $0$ like a Bessel process of dimension
$\beta + 1$ reflected at level $0$, and since
$\beta + 1>1$, the complication with the principal value
and the local time at zero does not occur;
see \cite[Chapter~10]{Yor97AspectsBMII}.
In particular, each $(\lambda_{j}(x))_{x\geq 0}$
is a semimartingale.
For $\beta=0$, 
$(\lambda(x)/\sqrt{2})_{x\geq 0}$ 
is just a reordered family of
$n$ i.i.d. standard Brownian motions.

\begin{rmk}
\label{Rmk beta Dyson}
We restrict to $\beta\geq 0$ because the case $\beta<0$ has not been considered in the literature. The problem is the extension of the process after a collision of $\lambda_{j}(x)$-s.
The collision of three or more particles, 
including all the $n$ together
for $\beta<-\frac{2(n-3)}{n(n-1)}$,
is no longer excluded.
However, we believe that the 
$\beta$-Dyson's Brownian motion can be defined
for all $\beta>-\frac{2}{n}$.
This is indeed the case if $n=2$.
One can use the reflected Bessel processes for that.
Let $(\rho(x))_{x\geq 0}$ be the Bessel process of 
dimension $\beta + 1$, reflected at level $0$,
satisfying away from $0$ the SDE
\begin{displaymath}
d\rho(x)= dW(x)+\dfrac{\beta}{2\rho(x)} dx,
\end{displaymath}
with $\rho(0)=0$.
The reflected version is precisely defined for
$\beta>-1=\frac{-2}{2}$;
see \cite[Section~XI.1]{RevuzYor1999BMGrundlehren} 
and  \cite[Section~3]{Lawler19Bessel}.
Let $(\widetilde{W}(x))_{x\geq 0}$ be
a standard Brownian motion starting from $0$,
independent from $(W(x))_{x\geq 0}$
Then, for $n=2$, one can construct the 
$\beta$-Dyson's Brownian motion as
\begin{equation}
\label{Eq b Dyson n=2 Bessel}
\lambda_{1}(x)=
\widetilde{W}(x) + \rho(x),
\qquad
\lambda_{2}(x)=
\widetilde{W}(x) - \rho(x).
\end{equation}
\end{rmk}

Next are some simple properties of the $\beta$-Dyson's Brownian motion.

\begin{prop}
\label{Prop b Dyson}
The following holds.
\begin{enumerate}
\item The process
$\big(\frac{1}{\sqrt{n}}p_{1}(\lambda(x))\big)_{x\geq 0}$
has the same law as $\phi_{\R_{+}}$.
\item The process $(\frac{1}{2}p_{2}(\lambda(x)))_{x\geq 0}$
is a square Bessel process of dimension $d(\beta,n)$
starting from $0$.
\item The processes
$(p_{1}(\lambda(x)))_{x\geq 0}$
and 
$\big(\lambda(x)
-\frac{1}{n}p_{1}(\lambda(x))\big)_{x\geq 0}$
are independent.
\item The process
$\big(\frac{1}{2}\big(p_{2}(\lambda(x))
-\frac{1}{n}p_{1}(\lambda(x))^{2}\big)\big)_{x\geq 0}$
is a square Bessel process of dimension $d(\beta,n)-1$
starting from $0$.
\end{enumerate}
\end{prop}

\begin{proof}
With Itô's formula, we get
\begin{displaymath}
d p_{1}(\lambda(x))=
\sqrt{2}\sum_{j=1}^{n}dW_{j}(x),
\end{displaymath}
\begin{displaymath}
d \frac{1}{2}p_{2}(\lambda(x))=
2\sum_{j=1}^{n}\dfrac{\lambda_{j}(x)}{\sqrt{2}}dW_{j}(x)
+d(\beta,n)dx,
\end{displaymath}
\begin{equation}
\label{Eq d p 2 c}
d\frac{1}{2}\Big(p_{2}(\lambda(x))-
\frac{1}{n}p_{1}(\lambda(x))^{2}\Big)=
2\sum_{j=1}^{n}
\dfrac{\lambda_{j}(x)-\frac{1}{n}p_{1}(\lambda(x))}
{\sqrt{2}}dW_{j}(x)
+(d(\beta,n)-1)dx,
\end{equation}
where the points $x\in\mathbb{R}_{+}$
for which $\lambda_{j}(x)=\lambda_{j-1}(x)$
for some $j\in\llbracket 2,n\rrbracket$ can be neglected.
This gives (1), (2) and (4) since the processes
\begin{displaymath}
d\widetilde{W}(x)=
\sum_{j=1}^{n}\dfrac{\lambda_{j}(x)}
{\sqrt{p_{2}(\lambda(x))}}dW_{j}(x),
~~\widetilde{W}(0)=0,
\end{displaymath}
and 
\begin{displaymath}
d\widecheck{W}(x)=
\sum_{j=1}^{n}\dfrac{\lambda_{j}(x)
-\frac{1}{n}p_{1}(\lambda(x))}
{\sqrt{p_{2}(\lambda(x))
-\frac{1}{n}p_{1}(\lambda(x))^{2}}}dW_{j}(x),
~~\widecheck{W}(0)=0,
\end{displaymath}
are both standard Brownian motions.
Again, one can neglect the points $x\in\mathbb{R}_{+}$
where $p_{2}(\lambda(x))
-\frac{1}{n}p_{1}(\lambda(x))^{2}=0$,
which only occur for $n=2$.

For (3), we have that
\begin{multline*}
d\Big(\lambda_{j}(x)-\dfrac{1}{n}p_{1}(\lambda(x))
\Big) = 
\sqrt{2} d\Big(W_{j}(x)-\frac{1}{n}p_{1}(W(x))\Big)
\\+ \beta \sum_{j'\neq j} 
\dfrac{dx}{\big(\lambda_{j}(x)
-\dfrac{1}{n}p_{1}(\lambda(x))\big)-
\big(\lambda_{j'}(x)
-\dfrac{1}{n}p_{1}(\lambda(x))\big)},
\end{multline*}
where
\begin{displaymath}
p_{1}(W(x)) = \sum_{j'=1}^{n} W_{j'}(x).
\end{displaymath}
The Brownian motion 
$p_{1}(W)=\frac{1}{\sqrt{2}}p_{1}(\lambda)$ 
is independent from the family of Brownian motions
$\big(W_{j}-\frac{1}{n}p_{1}(W)\big)_{1\leq j\leq n}$.
Further, the measurability of
$\big(\lambda_{j}-\frac{1}{n}p_{1}(\lambda)\big)_{1\leq j\leq n}$
with respect to
$\big(W_{j}-\frac{1}{n}p_{1}(W)\big)_{1\leq j\leq n}$
follows from the pathwise uniqueness of the solution to
\eqref{Eq b Dyson}; see \cite[Theorem~3.1]{CepaLepingle97DysonBM}.
\end{proof}

By combining Proposition \ref{Prop b Dyson} 
with Theorem \ref{Thm loop soup},
we get a first relation between the
$\beta$-Dyson's Brownian motion and 1D Brownian local times.
Compare it with Le Jan's isomorphism \eqref{Iso Le Jan 1D}.

\begin{cor}
\label{Cor Le Jan Dyson}
The process $\big(\frac{1}{2}p_{2}(\lambda(x))\big)_{x\geq 0}$
has the same law as the occupation field
$(L^{x}(\calL^{\alpha}_{\R_{+}}))_{x\geq 0}$
of a 1D Brownian loop soup $\calL^{\alpha}_{\R_{+}}$,
with the correspondence
\begin{equation}
\label{Eq alpha d b n}
2\alpha = d(\beta,n)=n+n(n-1)\dfrac{\beta}{2}.
\end{equation}
Further, let
$\calL^{\alpha-\frac{1}{2}}_{\R_{+}}$ and
$\widetilde{\calL}^{\frac{1}{2}}_{\R_{+}}$
be two independent 1D Brownian loop soups,
$\alpha$ still given by \eqref{Eq alpha d b n}.
Then, one has the following identity in law between pairs of processes:
\begin{displaymath}
\Big(\frac{1}{2}\Big(p_{2}(\lambda(x))
-\frac{1}{n}p_{1}(\lambda(x))^{2}\Big),
\frac{1}{2n}p_{1}(\lambda(x))^{2}\Big)_{x\geq 0}
\stackrel{\text{(law)}}{=}
(L^{x}(\calL^{\alpha-\frac{1}{2}}_{\R_{+}}),
L^{x}(\widetilde{\calL}^{\frac{1}{2}}_{\R_{+}}))_{x\geq 0}.
\end{displaymath}
\end{cor}

\subsection{Symmetric moments of $\beta$-Dyson's Brownian motion}
\label{Sec mome}

We will denote by $\langle\cdot\rangle_{\beta,n}^{\R_{+}}$
the expectation with respect to the
$\beta$-Dyson's Brownian motion \eqref{Eq b Dyson}.
This section will be devoted to deriving a recursive way to express the symmetric moments
\begin{equation}
\label{Eq mome}
\Big\langle \prod_{k=1}^{m(\nu)}
p_{\nu_{k}}(\lambda(x_{k}))
\Big\rangle_{\beta,n}^{\R_{+}} 
\end{equation}
for $\nu$ be a finite family of positive integers
with $\vert\nu\vert$ even
and $x_{1}\leq x_{2}\leq\dots\leq x_{m(\nu)}\in\R_{+}$.
This generalizes the Schwinger-Dyson equation \eqref{Eq SD beta}.
Note that if $\vert\nu\vert$ is odd then the moment equals $0$.

We will also use in the sequel the following notation. 
For $k\geq k'\in\N$, 
$\llbracket k,k'\rrbracket$ will denote the interval of integers
\begin{displaymath}
\llbracket k,k'\rrbracket =
\{ k,k+1,\dots,k'\}.
\end{displaymath}

We start by some lemmas.

\begin{lemma}
\label{Lem d p k}
Let $q\geq 3$. Then
\begin{eqnarray*}
dp_{q}(\lambda(x))&=&
q\sqrt{2}\sum_{j=1}^{n} \lambda_{j}(x)^{q-1}
dW_{j}(x) + 
\dfrac{\beta}{2} q \sum_{i=2}^{q-2} 
p_{i-1}(\lambda(x))
p_{q-1-i}(\lambda(x)) dx 
\\ && + 2 \dfrac{\beta}{2}n q p_{q-2}(\lambda(x)) dx
+\Big(1-\dfrac{\beta}{2}\Big)q(q-1)p_{q-2}(\lambda(x)) dx.
\end{eqnarray*}
\end{lemma}

\begin{proof}
By Itô's formula,
\begin{multline*}
dp_{q}(\lambda(x)) = 
q\sqrt{2}\sum_{j=1}^{n} \lambda_{j}(x)^{q-1}
dW_{j}(x) +
q(q-1)p_{q-2}(\lambda(x)) dx
\\+ \beta q \sum_{1\leq j<j'\leq n}
\dfrac{\lambda_{j}(x)^{q-1}-\lambda_{j'}(x)^{q-1}}{\lambda_{j}(x)-\lambda_{j'}(x)} dx.
\end{multline*}
But
\begin{multline*}
\sum_{1\leq j<j'\leq n}
\dfrac{\lambda_{j}(x)^{q-1}-\lambda_{j'}(x)^{q-1}}{\lambda_{j}(x)-\lambda_{j'}(x)} =
\sum_{1\leq j<j'\leq n}\sum^{q-2}_{r=0}
\lambda_{j}(x)^{r}\lambda_{j'}(x)^{q-2-r} 
\\
=\Big(n - \dfrac{q-1}{2}\Big)p_{q-2}(\lambda(x))+
\dfrac{1}{2} \sum_{i=2}^{q-2} 
p_{i-1}(\lambda(x))
p_{q-1-i}(\lambda(x)).
\qedhere
\end{multline*}
\end{proof}

\begin{lemma}
\label{Lem quad var k k prim}
Let $q,q'\geq 1$ with $q+q'>2$. Then
\begin{displaymath}
d\langle p_{q}(\lambda(x)),p_{q'}(\lambda(x))\rangle
= 2 q q' p_{q+q'-2}(\lambda(x)) dx.
\end{displaymath}
Moreover,
\begin{displaymath}
d\langle p_{1}(\lambda(x)),p_{1}(\lambda(x))\rangle
= 2 n dx.
\end{displaymath}
\end{lemma}

\begin{proof}
This is a straightforward computation.
\end{proof}

\begin{lemma}
\label{Lem true mart}
Let $\nu$ be a finite family of positive integers
and let $q\geq 0$. Then the process
\begin{equation}
\label{Eq loc mart}
\int_{0}^{x}p_{\nu}(\lambda(y))
\sum_{j=1}^{n} \lambda_{j}(y)^{q}
dW_{j}(y)
\end{equation}
is a martingale in the filtration of the
Brownian motions $((W_{j}(x))_{1\leq j\leq n})_{x\geq 0}$.
\end{lemma}

\begin{proof}
The process \eqref{Eq loc mart} is a local martingale.
Its quadratic variation is given by
\begin{displaymath}
\int_{0}^{x}
p_{\nu}(\lambda(y))^{2}
p_{2q}(\lambda(y)) dy .
\end{displaymath}
For every $y>0$, 
$\lambda(y)/\sqrt{2y}$
follows a fixed distribution,
which is up to reordering the G$\beta$E \eqref{Eq GbE}.
Thus,
\begin{displaymath}
\Big\langle \int_{0}^{x}
p_{\nu}(\lambda(y))^{2}
p_{2q}(\lambda(y)) dy\Big\rangle^{\R_{+}}_{\beta,n}
=
\langle p_{\nu}(\lambda)^{2}
p_{2q}(\lambda)\rangle_{\beta,n}
\int_{0}^{x}(2y)^{\vert\nu\vert + q} dy
< +\infty .
\end{displaymath}
So the quadratic variation is locally bounded in $\mathbb{L}^{1}$. 
It follows that \eqref{Eq loc mart} is a true martingale.
\end{proof}

Let $\nu$ be a finite family of positive integers.
and let $x_{1}\leq x_{2}\leq\dots\leq x_{m(\nu)}\in\R_{+}$.
For
$k\in\llbracket 1,m(\nu)\rrbracket$
and $x\geq x_{k-1}$,
let $f_{k}(x)$ denote the function
\begin{equation}
\label{Eq f k}
f_{k}(x):= 
\Big\langle \prod_{k'=1}^{k-1}
p_{\nu_{k'}}(\lambda(x_{k'}))
\prod_{k'=k}^{m(\nu)}
p_{\nu_{k'}}(\lambda(x))
\Big\rangle_{\beta,n}^{\R_{+}}.
\end{equation}
The main idea for expressing a symmetric moment \eqref{Eq mome}
is that for $x\geq x_{k-1}$,
the derivative $f'_{k}(x)$
is a linear combination of symmetric moments
of degree $\vert\nu\vert - 2$,
with coefficients depending on $\beta$ and $n$.
The precise expressions for these coefficients can be deduced from
Lemmas \ref{Lem d p k} and \ref{Lem quad var k k prim}.
Further, the moment \eqref{Eq mome}
equals $f_{m(\nu)}(x_{m(\nu)})$,
for every $k\in\llbracket 2,m(\nu)\rrbracket$,
$f_{k}(x_{k-1})=f_{k-1}(x_{k-1})$,
and
\begin{displaymath}
f_{1}(x_{1})= (2x_{1})^{\vert\nu\vert/2}
\langle p_{\nu}(\lambda)\rangle_{\beta,n},
\end{displaymath}
where $\langle p_{\nu}(\lambda)\rangle_{\beta,n}$ 
is the moment of the G$\beta$E, given by
Proposition \ref{Prop SD beta}.
So given the above initial conditions, and knowing the derivatives
$f'_{k}(x)$ one gets the moment \eqref{Eq mome}.
It turns out that this moment is a multivariate polynomial
in $(x_{k})_{1\leq k\leq m(\nu)}$.
Next we describe the recursion for this polynomial.

Let $(\Y_{kk})_{k\geq 1}$ denote a family of formal
commuting polynomials variables.
We will consider finite families of positive integers
$\nu=(\nu_{1},\nu_{2},\dots ,\nu_{m(\nu)})$
with $\vert\nu\vert$ even. 
The order of the $\nu_{k}$ will matter.
That is to say  we distinguish between $\nu$ and
$(\nu_{\sigma(1)},\nu_{\sigma(2)},\dots ,\nu_{\sigma(m(\nu))})$
for $\sigma$ a permutation of $\llbracket 1,m(\nu)\rrbracket$.
We want to construct a family of formal polynomials
$Q_{\nu,\beta,n}$ with parameters
$\nu$,$\beta$ and $n$, 
where $Q_{\nu,\beta,n}$ has for variables
$(\Y_{kk})_{1\leq k\leq m(\nu)}$.
To simplify the notations, 
we will drop the subscripts $\beta,n$ and just write $Q_{\nu}$.
The polynomials $Q_{\nu}$ will appear in the expression of the symmetric moments \eqref{Eq mome}.
We will denote by $c(\nu,\beta,n)$
the solutions to the recurrence
\eqref{Eq SD beta}, 
which for $\beta\in (-2/n,+\infty)$ are the moments 
$\langle p_{\nu}(\lambda)\rangle_{\beta,n}$.
By convention, $c((0),\beta,n)=n$
and $c(\emptyset,\beta,n)=1$.
For $k\geq 1$ and $Q$ a polynomial,
$Q^{k\leftarrow}$ will denote the polynomial in the variables
$(\Y_{k'k'})_{1\leq k'\leq k}$,
obtained from $Q$ by replacing each variable
$\Y_{k'k'}$ with $k'\geq k+1$
by the variable $\Y_{kk}$.
Note that $Q_{\nu}^{m(\nu)\leftarrow}=Q_{\nu}$
and that $Q_{\nu}^{1\leftarrow}$ is 
an univariate polynomial in $\Y_{11}$.
For $\Y$ a formal polynomial variable, 
$\deg_{\Y}$ will denote the partial degree in $\Y$.

\begin{defn}
\label{Def Rec Pol}
The family of polynomials $(Q_{\nu})_{\vert\nu\vert \text{ even}}$
is defined by the following.
\begin{enumerate}
\item $Q_{\nu}^{1\leftarrow}
=c(\nu,\beta,n)\Y_{1 1}^{\vert\nu\vert /2}$.
\item If $m(\nu)\geq 2$, 
then for every $k\in \llbracket 2,m(\nu)\rrbracket$,
\begin{eqnarray}
\label{Eq rec poly}
\dfrac{\partial}{\partial \Y_{k k}}
Q_{\nu}^{k\leftarrow}
&=&
\dfrac{\beta}{2}
\sum_{\substack{k\leq k'\leq m(\nu)\\\nu_{k'}>2}}
\dfrac{\nu(k')}{2}
\sum_{i=2}^{\nu_{k'}-2}
Q_{((\nu_{r})_{r\neq k'},i-1,\nu_{k'}-1-i)}
^{k\leftarrow}
\\
\nonumber
&&+
\dfrac{\beta}{2}n
\sum_{\substack{k\leq k'\leq m(\nu)\\\nu_{k'}>2}}
\nu(k')
Q_{((\nu_{r})_{r\neq k'},\nu_{k'}-2)}
^{k\leftarrow}
\\
\nonumber
&&+
\dfrac{\beta}{2}n^{2}
\sum_{\substack{k\leq k'\leq m(\nu)\\\nu_{k'}=2}}
Q_{(\nu_{r})_{r\neq k'}}
^{k\leftarrow}
\\
\nonumber
&&+
\Big(
1-\dfrac{\beta}{2}
\Big)
\sum_{\substack{k\leq k'\leq m(\nu)\\\nu_{k'}>2}}
\dfrac{\nu_{k'}(\nu_{k'}-1)}{2}
Q_{((\nu_{r})_{r\neq k'},\nu_{k'}-2)}
^{k\leftarrow}
\\
\nonumber
&&+
\Big(
1-\dfrac{\beta}{2}
\Big)
n
\sum_{\substack{k\leq k'\leq m(\nu)\\\nu_{k'}=2}}
Q_{(\nu_{r})_{r\neq k'}}
^{k\leftarrow}
\\
\nonumber
&&+
\sum_{\substack{k\leq k'<k''\leq m(\nu)\\\nu_{k'}+\nu_{k''}>2}}
\nu_{k'}\nu_{k''}
Q_{((\nu_{r})_{r\neq k',k''},\nu_{k'}+\nu_{k''}-2)}
^{k\leftarrow}
\\
\nonumber
&&+ n
\sum_{\substack{k\leq k'<k''\leq m(\nu)\\\nu_{k'}=\nu_{k''}=1}}
Q_{(\nu_{r})_{r\neq k',k''}}
^{k\leftarrow}.
\end{eqnarray}
If $k=m(\nu)$, then the last two lines of
\eqref{Eq rec poly} vanish.
\end{enumerate}
\end{defn}

Note that since the polynomials $Q_{\nu,\beta,n}$
are formal, one is not restricted by a specific range for $\beta$.
One could take any $\beta\in\mathbb{C}$ or even consider $\beta$ as a formal parameter. The specific range for $\beta$ will only matter when relating $Q_{\nu,\beta,n}$ to the symmetric moments of the
$\beta$-Dyson's Brownian motion.

\begin{prop}
\label{Prop defn Q}
Definition \ref{Def Rec Pol} uniquely defines a family of polynomials
$(Q_{\nu})_{\vert\nu\vert \text{ even}}$.
Moreover, the following properties hold.
\begin{enumerate}
\item For every $A$ monomial of $Q_{\nu}$
and every $k\in \llbracket 2,m(\nu)\rrbracket$,
\begin{equation}
\label{Eq monomial}
2\sum_{k\leq k'\leq m(\nu)}\deg_{\Y_{k' k'}}A
\leq\sum_{k\leq k'\leq m(\nu)}\nu_{k'},
\end{equation}
and
\begin{displaymath}
2\sum_{1\leq k'\leq m(\nu)}\deg_{\Y_{k' k'}}A = \vert\nu\vert.
\end{displaymath}
In particular, $Q_{\nu}$ is a homogeneous
polynomial of degree $\vert\nu\vert/2$.
\item For every $k\in \llbracket 1,m(\nu)\rrbracket$ and
every permutation $\sigma$ of 
$\llbracket k,m(\nu)\rrbracket$,
\begin{displaymath}
Q_
{(\nu_{r})_{1\leq r\leq k-1},
(\nu_{\sigma(r)})_{k\leq r\leq m(\nu)}}^{k\leftarrow}
=Q_{\nu}^{k\leftarrow}.
\end{displaymath}
\end{enumerate}
\end{prop}

\begin{proof}
The fact that the polynomials $Q_{\nu}$ are well defined can be proved
by induction on $\vert\nu\vert/2$.

For $\vert\nu\vert/2=1$, there are only two polynomials,
$Q_{(2)}$ and $Q_{(1,1)}$.
According to the condition (1),
\begin{displaymath}
Q_{(2)} = c((2),\beta,n)\Y_{11} = d(\beta,n)\Y_{11}
=\Big(\dfrac{\beta}{2}n^{2}+\Big(1-\dfrac{\beta}{2}\Big)n\Big)\Y_{11}.
\end{displaymath}
The condition (2) does not apply for $Q_{(2)}$.
For $Q_{(1,1)}$, according to the condition (2),
\begin{displaymath}
\dfrac{\partial}{\partial \Y_{22}} Q_{(1,1)} = 0.
\end{displaymath}
Thus, $Q_{(1,1)}$ contains no terms in $\Y_{22}$
and $Q_{(1,1)}=Q_{(1,1)}^{1\leftarrow}$.
From the condition (1) we further get
\begin{displaymath}
Q_{(1,1)} = c((1,1),\beta,n)\Y_{11} = n \Y_{11}.
\end{displaymath}

The induction step works as follows.
Assume $\vert\nu\vert/2\geq 2$.
The right hand side of \eqref{Eq rec poly}
involves only families of integers $\tilde{\nu}$ with
$\vert\tilde{\nu}\vert=\vert\nu\vert-2$.
According to the induction hypotheses,
$\dfrac{\partial}{\partial \Y_{k k}}
Q_{\nu}^{k\leftarrow}$ is uniquely determined
for every $k\in \llbracket 2,m(\nu)\rrbracket$.
Thus, for every $k\in \llbracket 2,m(\nu)\rrbracket$,
$Q_{\nu}^{k\leftarrow}-Q_{\nu}^{k\leftarrow}(\Y_{kk}=0)$
is uniquely determined. 
On top of that,
\begin{displaymath}
Q_{\nu}^{k\leftarrow}(\Y_{kk}=0)
=Q_{\nu}^{k-1\leftarrow} -
\big(Q_{\nu}^{k\leftarrow}-Q_{\nu}^{k\leftarrow}(\Y_{kk}=0)\big)
^{k-1\leftarrow}
.
\end{displaymath}
Moreover, by the condition (1), 
$Q_{\nu}^{1\leftarrow}$ is also uniquely determined.
Thus, all the polynomials
$(Q_{\nu}^{k\leftarrow})_{1\leq k\leq m(\nu)}$
are uniquely determined, with consistency by the
$Q\mapsto Q^{k\leftarrow}$ operations.
Finally, $Q_{\nu} = Q_{\nu}^{m(\nu)\leftarrow}$.

The properties (1) and (2) again follow easily by induction on
$\vert\nu\vert/2$.
\end{proof}

We are ready now to express the symmetric moments
\eqref{Eq mome}.

\begin{prop}
\label{Prop Mom b Dyson}
Let $\beta\geq 0$.
Let $\nu$ be a finite family of positive integers,
with $\vert\nu\vert$ even.
Let $Q_{\nu}=Q_{\nu,\beta,n}$
be the polynomial given by Definition \ref{Def Rec Pol}.
Let $x_{1}\leq x_{2}\leq\dots\leq x_{m(\nu)}\in\R_{+}$.
Then,
\begin{displaymath}
\Big\langle \prod_{k=1}^{m(\nu)}
p_{\nu_{k}}(\lambda(x_{k}))
\Big\rangle_{\beta,n}^{\R_{+}} = 
Q_{\nu}((\Y_{kk}=2x_{k})_{1\leq k\leq m(\nu)}).
\end{displaymath}
\end{prop}

\begin{proof}
The proof is done by induction on
$\vert\nu\vert/2$.

The case $\vert\nu\vert/2=1$ corresponds to
$\nu=(1,1)$ or $\nu=(2)$.
These are treated by Proposition \ref{Prop b Dyson},
and taking into account that the one-dimensional marginals of square Bessel processes follow Gamma distributions.

Now consider the induction step. Assume $\vert\nu\vert/2\geq 2$.
Recall the function $f_{k}(x)$ \eqref{Eq f k}
for $k\in\llbracket 1,m(\nu)\rrbracket$.
We have that
\begin{equation}
\label{Eq init}
f_{1}(x_{1})=c(\nu,\beta,n)(2x_{1})^{\vert\nu\vert/2}
=Q_{\nu}^{1\leftarrow}(\Y_{11}=2x_{1}),
\end{equation}
where for the second equality we applied the
condition (1) in Definition \ref{Def Rec Pol}.
If $m(\nu)=1$, there is nothing more to check.
In the case $m(\nu)\geq 2$, we need only to check that for
every $k\in\llbracket 2,m(\nu)\rrbracket$ and
every $x> x_{k-1}$,
\begin{eqnarray}
\label{Ed deriv}
f'_{k}(x)&=&
\dfrac{\partial}{\partial x}
Q_{\nu}^{k\leftarrow}
((\Y_{k'k'}=2x_{k'})_{1\leq k'\leq k-1},
\Y_{kk}=2x )
\\\nonumber &=&
2\Big(\dfrac{\partial}{\partial \Y_{kk}}
Q_{\nu}^{k\leftarrow}\Big)
((\Y_{k'k'}=2x_{k'})_{1\leq k'\leq k-1},
\Y_{kk}=2x ).
\end{eqnarray}
Indeed, given \eqref{Eq init}, by applying
\eqref{Ed deriv} to $k=2$, we further get
\begin{displaymath}
f_{2}(x_{2})=
P_{\nu}^{2\leftarrow}(\Y_{11}=2x_{1},\Y_{22}=2x_{2}),
\end{displaymath}
and by successively applying
\eqref{Ed deriv} to $k=3,\dots,k=m(\nu)$,
we at the end get
\begin{displaymath}
f_{m(\nu)}(x_{m(\nu)})=
Q_{\nu}^{m(\nu)\leftarrow}((\Y_{k'k'}=2x_{k'})_{1\leq k'\leq m(\nu)}),
\end{displaymath}
which is exactly what we want.
To show \eqref{Ed deriv}, we proceed as follows.
Let $(\F_{x})_{x\geq 0}$ be the filtration of the
Brownian motions 
$((W_{j}(x))_{1\leq j\leq n})_{x\geq 0}$.
Then, for $x>x_{k-1}$,
\begin{displaymath}
f_{k}(x)=
\Big\langle \prod_{k'=1}^{k-1}
p_{\nu_{k'}}(\lambda(x_{k'}))
\Big\langle
\prod_{k'=k}^{m(\nu)}
p_{\nu_{k'}}(\lambda(x))
\Big\vert\F_{x_{k-1}}\Big\rangle_{\beta,n}^{\R_{+}}
\Big\rangle_{\beta,n}^{\R_{+}},
\end{displaymath}
where $\langle\cdot\vert\F_{x_{k-1}}\rangle_{\beta,n}^{\R_{+}}$
denotes the conditional expectation.
To express
\begin{displaymath}
\Big\langle
\prod_{k'=k}^{m(\nu)}
p_{\nu_{k'}}(\lambda(x))
\Big\vert\F_{x_{k-1}}\Big\rangle_{\beta,n}^{\R_{+}},
\end{displaymath}
we apply Itô's formula to
\begin{displaymath}
\prod_{k'=k}^{m(\nu)}
p_{\nu_{k'}}(\lambda(x))
-\Big\langle
\prod_{k'=k}^{m(\nu)}
p_{\nu_{k'}}(\lambda(x_{k-1}))
\Big\rangle_{\beta,n}^{\R_{+}}.
\end{displaymath}
The local martingale part is, according to Lemma
\ref{Lem true mart}, a true martingale, and thus gives a
$0$ conditional expectation. The bounded variation part is a linear combination of terms of form 
$p_{\tilde{\nu}}(\lambda(x)) dx$, with
\begin{displaymath}
\vert\tilde{\nu}\vert = 
\Big(\sum_{k'=k}^{m(\nu)} \nu_{k'}\Big)-2,
\end{displaymath}
the exact expressions following from
Lemma \ref{Lem d p k} and Lemma \ref{Lem quad var k k prim}.
By comparing these expressions with the 
recurrence \eqref{Eq rec poly}, 
and using the induction hypothesis at the step $\vert\nu\vert/2-1$,
we get \eqref{Ed deriv}.
At this stage we omit detailing the tedious but completely elementary computations.
\end{proof}

\subsection{More general formal polynomials}
\label{Subsec more poly}

In previous Section \ref{Sec mome}, 
we defined recursively a family of formal polynomials
$Q_{\nu}=Q_{\nu,\beta,n}$
(Definition \ref{Def Rec Pol}),
which encode the symmetric moments of the $\beta$-Dyson's Brownian motion (Proposition \ref{Prop Mom b Dyson}).
However, these polynomials are insufficient 
both for the generalization of the
BFS-Dynkin isomorphism
(forthcoming Proposition \ref{Prop Iso b Dyson})
and for expressing the symmetric moments of
the stationary version of the $\beta$-Dyson's Brownian motion
(forthcoming Proposition \ref{Prop Iso stat b Dyson}).
Therefore we introduce an other family of
formal polynomials $P_{\nu}=P_{\nu,\beta,n}$,
with $P_{\nu}$ constructed out of $Q_{\nu}$
in a straightforward way which we describe next.

On top of the formal commuting polynomial variables
$(\Y_{kk})_{k\geq 1}$ appearing in the polynomials
$Q_{\nu}$,
we also consider the family of the formal commuting variables
$(\cY_{k-1\,k})_{k\geq 2}$, also commuting with the first one.
A polynomial $P_{\nu}$ will have for variables
$(\Y_{kk})_{1\leq k\leq m(\nu)}$
and $(\cY_{k-1\,k})_{2\leq k\leq m(\nu)}$.

\begin{defn}
\label{Def P nu}
Given $\nu$ a finite family of positive integers
with $\vert\nu\vert$ even,
let $P_{\nu}$ be the polynomial in the variables
$(\Y_{kk})_{1\leq k\leq m(\nu)},(\cY_{k-1\,k})_{2\leq k\leq m(\nu)}$
defined by the following.
\begin{enumerate}
\item $P_{\nu}((\Y_{kk})_{1\leq k\leq m(\nu)},(\cY_{k-1\,k}=1)_{2\leq k\leq m(\nu)}) =
Q_{\nu}((\Y_{kk})_{1\leq k\leq m(\nu)})$.
\item For every $A$ monomial of $P_{\nu}$
and every $k\in \llbracket 2,m(\nu)\rrbracket$,
\begin{equation}
\label{Eq deg A}
\deg_{\cY_{k-1\,k}}A +
2\sum_{k\leq k'\leq m(\nu)}\deg_{\Y_{k' k'}}A
= \sum_{k\leq k'\leq m(\nu)}\nu_{k'}.
\end{equation}
\end{enumerate}
\end{defn}

The property \eqref{Eq monomial} ensures that 
$P_{\nu}=P_{\nu,\beta,n}$
is well defined.
As for $Q_{\nu,\beta,n}$, 
$P_{\nu,\beta,n}$ is defined for every $\beta\in\mathbb{C}$.

Proposition \ref{Prop Mom b Dyson} and Definition \ref{Def P nu}
immediately imply the following.

\begin{cor}
\label{Cor mome P nu}
Let $\beta\geq 0$.
Let $\nu$ be a finite family of positive integers,
with $\vert\nu\vert$ even.
Let $x_{1}\leq x_{2}\leq\dots\leq x_{m(\nu)}\in\R_{+}$.
Then,
\begin{multline*}
\Big\langle \prod_{k=1}^{m(\nu)}
p_{\nu_{k}}(\lambda(x_{k}))
\Big\rangle_{\beta,n}^{\R_{+}} = 
P_{\nu}((\Y_{kk}=2x_{k})_{1\leq k\leq m(\nu)},
(\cY_{k-1\,k}=1)_{2\leq k\leq m(\nu)}
)
\\ =
P_{\nu}((\Y_{kk}=G_{\R_{+}}(x_{k},x_{k}))_{1\leq k\leq m(\nu)},
(\cY_{k-1\,k}=G_{\R_{+}}(x_{k-1},x_{k})/G_{\R_{+}}(x_{k-1},x_{k-1}))_{2\leq k\leq m(\nu)}
).
\end{multline*}
\end{cor}

Next are the expressions for
$Q_{(1,1,\dots,1)}$, $P_{(1,1,\dots,1)}$, $Q_{(2,2,\dots,2)}$
and $P_{(2,2,\dots,2)}$.

\begin{prop}
\label{Prop Gauss Perm}
Let $m\in\N\setminus\{0\}$.
Let $\mathsf{M}=(\mathsf{M}_{kk'})_{1\leq k,k'\leq m}$
be the formal symmetric matrix with entries given by
\begin{equation}
\label{Eq M k k prim}
\mathsf{M}_{kk} = \Y_{kk},
\qquad
\text{for } k<k',~
\mathsf{M}_{kk'} = \mathsf{M}_{k'k} 
= \Y_{kk}\prod_{k+1\leq r\leq k'}\cY_{r-1\,r}.
\end{equation}
The following holds.
\begin{enumerate}
\item Assume $m$ is even,
and let $\nu=(1,1,\dots,1)$, where $1$ appears $m$ times.
Then $Q_{(1,1,\dots,1)}$ $P_{(1,1,\dots,1)}$ satisfies the Wick's rule for Gaussians:
\begin{displaymath}
Q_{(1,1,\dots,1)} = n^{\frac{m}{2}}
\sum_{\substack{(\{a_{i},b_{i}\})_{1\leq i\leq m/2}\\
\text{partition in pairs}
\\
\text{of } \llbracket 1,m\rrbracket
}}
\prod_{i=1}^{m/2}
\Y_{a_{i}\wedge b_{i}\,a_{i}\wedge b_{i}},
\qquad
P_{(1,1,\dots,1)} = n^{\frac{m}{2}}
\sum_{\substack{(\{a_{i},b_{i}\})_{1\leq i\leq m/2}\\
\text{partition in pairs}
\\
\text{of } \llbracket 1,m\rrbracket
}}
\prod_{i=1}^{m/2}
\mathsf{M}_{a_{i}b_{i}},
\end{displaymath}
where $a_{i}\wedge b_{i} =\min(a_{i},b_{i})$ and
where the sums run over the $m!/(2^{\frac{m}{2}}(m/2)!)$
partitions in pairs.
\item Let $\nu=(2,2,\dots,2)$, where $2$ appears $m$ times.
Then
\begin{displaymath}
Q_{(2,2,\dots,2)} = 
2^{m}\Perm_{d(\beta,n)/2}
((\Y_{k\wedge k'\,k\wedge k'})_{1\leq k, k'\leq m}),
\qquad
P_{(2,2,\dots,2)} = 
2^{m}\Perm_{d(\beta,n)/2}(\mathsf{M}).
\end{displaymath}
\end{enumerate}
\end{prop}

\begin{proof}
The expressions for $Q_{(1,1,\dots,1)}$ and $Q_{(2,2,\dots,2)}$ 
are easily obtained by induction on $m$
using Definition \ref{Def Rec Pol}.
Alternatively, for $\beta\geq 0$,
one can use that under the law of $\beta$-Dyson's Brownian motion,
the process
$(p_{1}(\lambda(x)))_{x\geq 0}$
is Gaussian and
the process $(p_{2}(\lambda(x)))_{x\geq 0}$
is $d(\beta,n)/2$-permanental;
see Proposition \ref{Prop b Dyson}.
This gives the expression of 
$Q_{(1,1,\dots,1)}$ and $Q_{(2,2,\dots,2)}$
for $\beta\geq 0$. 
To extend it to general $\beta$ one can use that the coefficients of the polynomials $Q_{\nu}$ are themselves polynomials in $\beta$.
The expressions for $P_{(1,1,\dots,1)}$ and $P_{(2,2,\dots,2)}$
are immediately deducible from those for
$Q_{(1,1,\dots,1)}$ and $Q_{(2,2,\dots,2)}$
by following Definition \ref{Def P nu}.
\end{proof}

For other examples of $P_{\nu}$, see the Appendix.

As a side remark, we observe next that the value $\beta=-\frac{2}{n}$
plays a special role for the polynomials
$Q_{\nu,\beta,n}$ and $P_{\nu,\beta,n}$.
In particular, $P_{\nu,\beta=-\frac{2}{n},n}$ gives the moments 
of the stochastic processes $(\phi_{\R_{+}}(x))_{x\geq 0}$
and $(\phi_{K}(x))_{x\in \R}$ introduced in Section \ref{Sec Iso 1D},
which are Gaussian.
This is also related to the fact that
in the limit $\beta\to -\frac{2}{n}$,
the G$\beta$E converges in law to $n$ identical Gaussians
\eqref{Eq n ident Gauss}.

\begin{prop}
\label{Prop beta 2 n}
Let $n\geq 1$. Let $K>0$.
Let $\nu$ be a finite family of positive integers with
$\vert\nu\vert$ even.
Let $x_{1}\leq\dots\leq x_{m(\nu)}$ be $m(\nu)$ points in
$(0,+\infty)$, resp. in $\R$. Then
\begin{multline*}
Q_{\nu,\beta=-\frac{2}{n},n}
((\Y_{kk}=2x_{k})_{1\leq k\leq m(\nu)}) = \\
P_{\nu,\beta=-\frac{2}{n},n}
((\Y_{kk}=2x_{k})_{1\leq k\leq m(\nu)},
(\cY_{k-1\,k}=1)_{2\leq k\leq m(\nu)})
=n^{m(\nu)-\vert\nu\vert/2}\E\Big[\prod_{k=1}^{m(\nu)}\phi_{\R_{+}}(x_{k})^{\nu_{k}}\Big],
\end{multline*}
resp.
\begin{multline*}
P_{\nu,\beta=-\frac{2}{n},n}
((\Y_{kk}=1/\sqrt{2K})_{1\leq k\leq m(\nu)},
(\cY_{k-1\,k}=e^{-\sqrt{2K}(x_{k}-x_{k-1})})_{2\leq k\leq m(\nu)})
\\=n^{m(\nu)-\vert\nu\vert/2}\E\Big[\prod_{k=1}^{m(\nu)}\phi_{K}(x_{k})^{\nu_{k}}\Big].
\end{multline*}
That is to say, the variables
$\Y_{kk}$ are replaced by
$G_{\R_{+}}(x_{k},x_{k})$, resp. $G_{K}(x_{k},x_{k})$,
and the variables $\cY_{k-1\,k}$ by
$G_{\R_{+}}(x_{k-1},x_{k})/G_{\R_{+}}(x_{k-1},x_{k-1})$,
resp.
$G_{K}(x_{k-1},x_{k})/G_{K}(x_{k-1},x_{k-1})$.
\end{prop}

\begin{proof}
First, one can check that
\begin{equation}
\label{Eq sol 2 n}
c\Big(\nu,\beta=-\frac{2}{n},n\Big)=
n^{m(\nu)-\vert\nu\vert/2}
\dfrac{\vert\nu\vert !}{2^{\vert\nu\vert/2} (\vert\nu\vert/2)!}.
\end{equation}
This follows from Proposition \ref{Proberty GbE}.
The key point is that
\begin{displaymath}
d\Big(\beta=-\frac{2}{n},n\Big)=1.
\end{displaymath}

Given $\nu$ a finite family of positive integers, 
let $\mathbf{k}_{\nu} : \llbracket 1,\vert\nu\vert\rrbracket
\mapsto \llbracket 1,m(\nu)\rrbracket$ be the function such that
\begin{equation}
\label{Eq k nu}
\mathbf{k}_{\nu}^{-1}(1)=
\llbracket 1,\nu_{1}\rrbracket,
\qquad \text{for } k'\in \llbracket 2,m(\nu)\rrbracket,~
\mathbf{k}_{\nu}^{-1}(k')=
\llbracket \nu_{1}+\dots+\nu_{k'-1}+1,
\nu_{1}+\dots+\nu_{k'}\rrbracket.
\end{equation}
Further, let $(\widetilde{Q}_{\nu})_{\vert\nu\vert \text{ even}}$ 
be the following formal polynomials:
\begin{displaymath}
\widetilde{Q}_{\nu}= n^{m(\nu)-\vert\nu\vert/2}
\sum_{\substack{(\{a_{i},b_{i}\})_{1\leq i\leq \vert\nu\vert/2}\\
\text{partition in pairs}
\\
\text{of } \llbracket 1,\vert\nu\vert\rrbracket
}}
\prod_{i=1}^{\vert\nu\vert/2}
\Y_{\mathbf{k}_{\nu}(a_{i})\wedge\mathbf{k}_{\nu}(b_{i})\,\mathbf{k}_{\nu}(a_{i})\wedge\mathbf{k}_{\nu}(b_{i})}.
\end{displaymath}
To conclude, we need only to check that
$\widetilde{Q}_{\nu} = Q_{\nu,\beta=-\frac{2}{n},n}$
for all $\nu$ with $\vert\nu\vert$ even.
Indeed, this immediately implies that
\begin{displaymath}
P_{\nu,\beta=-\frac{2}{n},n}= n^{m(\nu)-\vert\nu\vert/2}
\sum_{\substack{(\{a_{i},b_{i}\})_{1\leq i\leq \vert\nu\vert/2}\\
\text{partition in pairs}
\\
\text{of } \llbracket 1,\vert\nu\vert\rrbracket
}}
\prod_{i=1}^{\vert\nu\vert/2}
\mathsf{M}_{\mathbf{k}_{\nu}(a_{i})\mathbf{k}_{\nu}(b_{i})},
\end{displaymath}
where the $\mathsf{M}_{k k'}$ are given by \eqref{Eq M k k prim},
and thus $n^{-m(\nu)+\vert\nu\vert/2}P_{\nu,\beta=-\frac{2}{n},n}$
corresponds to the Wick's rule.
So by evaluating in $\Y_{kk}=G_{\R_{+}}(x_{k},x_{k})$
and $\cY_{k-1\,k}=G_{\R_{+}}(x_{k-1},x_{k})/G_{\R_{+}}(x_{k-1},x_{k-1})$,
resp.
$\Y_{kk}=G_{K}(x_{k},x_{k})$
and $\cY_{k-1\,k}=G_{K}(x_{k-1},x_{k})/G_{K}(x_{k-1},x_{k-1})$,
one gets the moments of $\phi_{\R_{+}}$,
resp. $\phi_{K}$.

The identity $\widetilde{Q}_{\nu} = Q_{\nu,\beta=-\frac{2}{n},n}$
can be checked by induction over
$\vert\nu\vert/2$ by following
Definition \ref{Def Rec Pol}.
From \eqref{Eq sol 2 n} follows that the $\widetilde{Q}_{\nu}$
satisfy the condition (1) in Definition \ref{Def Rec Pol}.
One can further check the recurrence \eqref{Eq rec poly},
and this amounts to counting the pairs in
$\mathbf{k}_{\nu}^{-1}(\llbracket k,m(\nu)\rrbracket)$.
\end{proof}

\subsection{BFS-Dynkin isomorphism for $\beta$-Dyson's Brownian motion}
\label{Sec Dynkin Dyson}

We will denote by $\Upsilon$ a generic finite family of
continuous paths on $\mathbb{R}$,
$\Upsilon=(\gamma_{1},\dots,\gamma_{J})$,
and $J(\Upsilon)$ will denote the size $J$ of the family.
We will consider finite Brownian measures on $\Upsilon$
where $J(\Upsilon)$ is not fixed but may take several values under the measure. Given $x\in \R$, $L^{x}(\Upsilon)$ will denote the sum of Brownian local times at $x$:
\begin{displaymath}
L^{x}(\Upsilon)=\sum_{i=1}^{J(\Upsilon)}L^{x}(\gamma_{i}).
\end{displaymath}
$L(\Upsilon)$ will denote the occupation field
$x\mapsto L^{x}(\Upsilon)$.

Given $\nu$ a finite family of positive integers with
$\vert\nu\vert$ even and 
$0<x_{1}< x_{2}<\dots<x_{m(\nu)}$,
$\mu_{\R_{+}}^{\nu,x_{1},\dots,x_{m(\nu)}}(d\Upsilon)$
(also depending on $\beta$ and $n$)
will be the measure on finite families of
continuous paths obtained by substituting
in the polynomial $P_{\nu}=P_{\nu,\beta,n}$
for each variable $\Y_{kk}$ the measure
$\mu_{\R_{+}}^{x_{k},x_{k}}$,
and for each variable $\cY_{k-1\,k}$
the measure $\check{\mu}_{\R_{+}}^{x_{k-1},x_{k}}$;
see Section \ref{Sec Iso 1D}.
Since we will deal with the functional
$L(\Upsilon)$ under 
$\mu_{\R_{+}}^{\nu,x_{1},\dots,x_{m(\nu)}}(d\Upsilon)$,
the order of the Brownian measures in a product will not matter.
For instance, for
$\nu=(2,1,1)$ (see Appendix),
\begin{displaymath}
P_{(2,1,1)}=
\Big(\dfrac{\beta}{2} n^{3} + \Big(1-\dfrac{\beta}{2}\Big) n^{2}\Big)
\Y_{11}\Y_{22}\cY_{23}
+ 2n \Y_{11}^{2}\cY_{12}^{2}\cY_{23},
\end{displaymath}
and
\begin{eqnarray*}
\mu_{\R_{+}}^{(2,1,1),x_{1},x_{2},x_{3}}&=&
\Big(\dfrac{\beta}{2} n^{3} + \Big(1-\dfrac{\beta}{2}\Big) n^{2}\Big)
\mu_{\R_{+}}^{x_{1},x_{1}}\otimes\mu_{\R_{+}}^{x_{2},x_{2}}
\otimes\check{\mu}_{\R_{+}}^{x_{2},x_{3}}
\\&&
+2n\mu_{\R_{+}}^{x_{1},x_{1}}\otimes\mu_{\R_{+}}^{x_{1},x_{1}}
\otimes\check{\mu}_{\R_{+}}^{x_{1},x_{2}}
\otimes\check{\mu}_{\R_{+}}^{x_{1},x_{2}}
\otimes\check{\mu}_{\R_{+}}^{x_{2},x_{3}}.
\end{eqnarray*}
Note that depending on values of $n$ and $\beta$,
a measure
$\mu_{\R_{+}}^{\nu,x_{1},\dots,x_{m(\nu)}}$
may be signed.

Next is a version of BFS-Dynkin isomorphism
(Theorem \eqref{ThmIsoDynkin}) for
$\beta$-Dyson's Brownian motion.

\begin{prop}
\label{Prop Iso b Dyson}
Let $\nu$ be a finite family of positive integers,
with $\vert\nu\vert$ even and let
$0<x_{1}< x_{2}<\dots<x_{m(\nu)}$.
Let $F$ be a bounded measurable functional on
$\mathcal{C}(\R_{+})$.
Then
\begin{equation}
\label{Eq Iso b Dyson}
\Big\langle
\prod_{k=1}^{m(\nu)}
p_{\nu_{k}}(\lambda(x_{k}))
F\Big(\dfrac{1}{2}p_{2}(\lambda)\Big)
\Big\rangle_{\beta,n}^{\R_{+}}
=\int_{\Upsilon}
\Big\langle
F\Big(\dfrac{1}{2}p_{2}(\lambda)+L(\Upsilon)\Big)
\Big\rangle_{\beta,n}^{\R_{+}}
\mu_{\R_{+}}^{\nu,x_{1},\dots,x_{m(\nu)}}(d\Upsilon).
\end{equation}
\end{prop}

\begin{rmk}
In the limiting case when $x_{k}=x_{k-1}$ for some
$k\in\llbracket 2,m(\nu)\rrbracket$,
$\cY_{k-1\,k}$ in $P_{\nu}$ has to be replaced by the constant $1$ instead of a measure on Brownian paths.
\end{rmk}

\begin{rmk}
For $\beta\in\{0,1,2,4\}$,
\eqref{Eq Iso b Dyson} reduces to the Gaussian case of
Theorem \ref{ThmIsoDynkin}.
\end{rmk}

Let us first outline our strategy for proving
Proposition \ref{Prop Iso b Dyson}.
By density arguments it is enough to show \eqref{Eq Iso b Dyson} for functionals $F$ of form
\begin{displaymath}
F((\ell(x))_{x\geq 0})
=\exp\Big(-\int_{\R_{+}}\ell(x)\chi(x) dx\Big),
\end{displaymath}
where $\chi$ is a continuous non-negative function with compact support in $(0,+\infty)$.
For such $F$, the value returned by the right-hand side of
\eqref{Eq Iso b Dyson} is well understood and is related to the local times of Brownian motions with a killing rate given by $\chi$.
In order to deal with the left-hand side of \eqref{Eq Iso b Dyson},
one interprets
\begin{displaymath}
\dfrac{\exp\Big(
-\frac{1}{2}\int_{0}^{+\infty}
p_{2}(\lambda(y))\chi(y) dy
\Big)}
{\Big\langle\exp\Big(
-\frac{1}{2}\int_{0}^{+\infty}
p_{2}(\lambda(y))\chi(y) dy
\Big)\Big\rangle_{\beta,n}^{\R_{+}}}
\end{displaymath}
as a density in a change of measure.
Then it remains to describe the law of the stochastic process
$(\lambda(x))_{x\geq 0}$ under the new measure, and in particular
express its symmetric moments.
It turns out that under the new measure,
the process can still be reduced to a $\beta$-Dyson's Brownian motion
through a deterministic transformation reminiscent of
the scale and time changes for one-dimensional diffusions;
see Lemma \ref{Lem scale change}.

We start by some intermediate lemmas.
Recall that $(\F_{x})_{x\geq 0}$ denotes the filtration of the
Brownian motions 
$((W_{j}(x))_{1\leq j\leq n})_{x\geq 0}$
in \eqref{Eq b Dyson}.
Consider $\chi$ a continuous non-negative function with compact support in $(0,+\infty)$.
Let $u_{\chi\downarrow}$ denote the unique solution to
\begin{displaymath}
\dfrac{1}{2}\dfrac{d^{2}}{dx}u
=\chi u
\end{displaymath}
which is positive non-increasing on $\R_{+}$,
with $u_{\chi\downarrow}(0)=1$.
See \cite[Section~2.1]{Lupu20131dimLoops} for details.
Then
\begin{displaymath}
u_{\chi\downarrow}(+\infty)
=\lim_{x\to +\infty}u_{\chi\downarrow}(x) > 0.
\end{displaymath}

\begin{lemma}
\label{Lem chi}
Let $\D_{\chi}(+\infty)$ be the positive r.v.
\begin{equation}
\label{Eq Density}
\D_{\chi}(+\infty):=
u_{\chi\downarrow}(+\infty)^{-\frac{1}{2}d(\beta,n)}
\exp\Big(
-\dfrac{1}{2}\int_{0}^{+\infty}
p_{2}(\lambda(y))\chi(y) dy
\Big).
\end{equation}
Then $\langle\D_{\chi}(+\infty)\rangle^{\R_{+}}_{\beta,n}=1$.
Moreover,
\begin{multline}
\label{Eq Density cond}
\D_{\chi}(x):=
\langle\D_{\chi}(+\infty)
\vert\F_{x}\rangle^{\R_{+}}_{\beta,n}
\\
=
u_{\chi\downarrow}(x)^{-\frac{1}{2}d(\beta,n)}
\exp\Big(
-\dfrac{1}{2}\int_{0}^{x}
p_{2}(\lambda(y))\chi(y) dy
\Big)
\exp\Big(
\dfrac{1}{4}p_{2}(\lambda(x))
\dfrac{u'_{\chi\downarrow}(x)}{u_{\chi\downarrow}(x)}
\Big).
\end{multline}
Let
\begin{displaymath}
\M_{\chi}(x):=\dfrac{1}{\sqrt{2}}
\int_{0}^{x}
\dfrac{u'_{\chi\downarrow}(y)}{u_{\chi\downarrow}(y)}
\sum_{j=1}^{n}
\lambda_{j}(y) dW_{j}(y) .
\end{displaymath}
Then $(\M_{\chi}(x))_{x\geq 0}$ is a martingale with respect to the filtration $(\F_{x})_{x\geq 0}$ and
for all $x\geq 0$,
\begin{displaymath}
\D_{\chi}(x)=\exp\Big(\M_{\chi}(x)
-\dfrac{1}{2}\langle\M_{\chi},\M_{\chi}\rangle(x)\Big).
\end{displaymath}
\end{lemma}

\begin{proof}
\eqref{Eq Density} and \eqref{Eq Density cond} follow from the properties of square Bessel processes.
See Theorem (1.7), Section XI.1 in \cite{RevuzYor1999BMGrundlehren}.
$(\M_{\chi}(x))_{x\geq 0}$ is obviously a (true) martingale,
as can be seen with the quadratic variation.
Further,
\begin{displaymath}
d\Big(\dfrac{1}{4}p_{2}(\lambda(x))
\dfrac{u'_{\chi\downarrow}(x)}{u_{\chi\downarrow}(x)}\Big)
= d\M_{\chi}(x) 
+ \dfrac{1}{2}p_{2}(\lambda(x))\chi(x) dx
- \dfrac{1}{4}p_{2}(\lambda(x))
\dfrac{u'_{\chi\downarrow}(x)^{2}}{u_{\chi\downarrow}(x)^{2}} dx
+\dfrac{1}{2}d(\beta,n)
\dfrac{u'_{\chi\downarrow}(x)}{u_{\chi\downarrow}(x)} dx,
\end{displaymath}
and
\begin{displaymath}
d\dfrac{1}{2}\langle\M_{\chi},\M_{\chi}\rangle(x)
=\dfrac{1}{4}p_{2}(\lambda(x))
\dfrac{u'_{\chi\downarrow}(x)^{2}}{u_{\chi\downarrow}(x)^{2}} dx.
\end{displaymath}
Thus
\begin{displaymath}
d\Big(\M_{\chi}(x) - \dfrac{1}{2}\langle\M_{\chi},\M_{\chi}\rangle(x)\Big)
= d\log (\D_{\chi}(x)).
\qedhere
\end{displaymath}
\end{proof}

\begin{lemma}
\label{Lem tilde lambda}
Let  be
$(\tilde{\lambda}(x)=
(\tilde{\lambda}_{1}(x),\dots,
\tilde{\lambda}_{n}(x)))_{x\geq 0}$
with $\tilde{\lambda}_{1}(x)\geq\dots\geq\tilde{\lambda}_{n}(x)$,
satisfying the SDE
\begin{equation}
\label{Eq b Dyson tilde}
d\tilde{\lambda}_{j}(x)=
\sqrt{2} dW_{j}(x) + 
\dfrac{u'_{\chi\downarrow}(x)}{u_{\chi\downarrow}(x)}
\tilde{\lambda}_{j}(x) dx
+ \beta \sum_{j'\neq j} 
\dfrac{dx}{\tilde{\lambda}_{j}(x)-\tilde{\lambda}_{j'}(x)},
\end{equation}
with initial condition $\tilde{\lambda}(0)=0$.
Further consider a change of measure with density
$\D_{\chi}(+\infty)$ \eqref{Eq Density}
on the filtered probability space with filtration
$(\F_{x})_{x\geq 0}$.
Then $\lambda$ after the change of measure and 
$\tilde{\lambda}$ before the change of measure have the same law.
\end{lemma}

\begin{proof}
The existence and uniqueness of strong solutions to
\eqref{Eq b Dyson tilde} is given by 
\cite[Theorem~3.1]{CepaLepingle97DysonBM}.
The rest is a consequence of Girsanov's theorem;
see Theorems (1.7) and (1.12),
Section VIII.1, in \cite{RevuzYor1999BMGrundlehren}.
Indeed,
\begin{displaymath}
d\langle W_{j}(x),\M_{\chi}(x)\rangle
=\dfrac{1}{\sqrt{2}}
\dfrac{u'_{\chi\downarrow}(x)}{u_{\chi\downarrow}(x)}
\lambda_{j}(x) dx.
\end{displaymath}
Thus, after the change of measure,
the
\begin{displaymath}
W_{j}(x)-\dfrac{1}{\sqrt{2}}\int_{0}^{x}
\dfrac{u'_{\chi\downarrow}(y)}{u_{\chi\downarrow}(y)}
\lambda_{j}(y) dy
\end{displaymath}
for $j\in\llbracket 1,n\rrbracket$ are $n$ i.i.d. standard Brownian motions.
\end{proof}

Let $\psi_{\chi}$ denote the following diffeomorphism of
$\R_{+}$:
\begin{displaymath}
\psi_{\chi}(x)=
\int_{0}^{x}\dfrac{dy}{u_{\chi\downarrow}(y)^{2}}.
\end{displaymath}
Let $\psi_{\chi}^{-1}$ be the inverse diffeomorphism.

\begin{lemma}
\label{Lem scale change}
If $\tilde{\lambda}$ is a solution to the SDE 
\eqref{Eq b Dyson tilde}, then
the process
\begin{displaymath}
\Big(
\dfrac{1}{u_{\chi\downarrow}(\psi_{\chi}^{-1}(x))}
\tilde{\lambda}(\psi_{\chi}^{-1}(x))
\Big)_{x\geq 0}
\end{displaymath}
satisfies the SDE \eqref{Eq b Dyson}.
\end{lemma}

\begin{proof}
The process $\Big(\dfrac{1}{u_{\chi\downarrow}(x)}
\tilde{\lambda}(x)\Big)_{x\geq 0}$
satisfies
\begin{displaymath}
d\Big(\dfrac{1}{u_{\chi\downarrow}(x)}
\tilde{\lambda}_{j}(x)\Big)
=\dfrac{\sqrt{2}}{u_{\chi\downarrow}(x)} dW_{j}(x)
+
\beta\sum_{j'\neq j}
\dfrac{1}
{u_{\chi\downarrow}(x)^{-1}\tilde{\lambda}_{j}(x)
-u_{\chi\downarrow}(x)^{-1}\tilde{\lambda}_{j'}(x)}
\dfrac{dx}{u_{\chi\downarrow}(x)^{2}}.
\end{displaymath}
By further performing the change of variable given by
$\psi_{\chi}$, one gets \eqref{Eq b Dyson}.
\end{proof}

In the sequel
$(G_{\R_{+},\chi}(x,y))_{x,y\geq 0}$ will denote the
Green's function of
$\frac{1}{2}\frac{d^{2}}{dx^{2}}-\chi$ on $\R_{+}$
with condition 0 in 0.
Then for $0\leq x\leq y$,
\begin{equation}
\label{Eq Green psi}
G_{\R_{+},\chi}(x,y) =
2u_{\chi\downarrow}(x)
\psi_{\chi}(x)u_{\chi\downarrow}(y).
\end{equation}
Indeed,
\begin{displaymath}
\dfrac{1}{2}\dfrac{\partial^{2}}{\partial y^{2}}
\Big(2u_{\chi\downarrow}(x)
\psi_{\chi}(x)u_{\chi\downarrow}(y)
\Big)
=\chi(y)\Big(2u_{\chi\downarrow}(x)
\psi_{\chi}(x)u_{\chi\downarrow}(y)
\Big),
\end{displaymath}
\begin{eqnarray*}
\dfrac{1}{2}\dfrac{\partial^{2}}{\partial x^{2}}
\Big(
2u_{\chi\downarrow}(x)
\psi_{\chi}(x)u_{\chi\downarrow}(y)
\Big)
&=&
\dfrac{1}{2}\dfrac{\partial}{\partial x}
\Big(
2u_{\chi\downarrow}'(x)
\psi_{\chi}(x)u_{\chi\downarrow}(y)
+2\dfrac{u_{\chi\downarrow}(y)}{u_{\chi\downarrow}(x)}
\Big)
\\&=&
\chi(x)\Big(2u_{\chi\downarrow}(x)
\psi_{\chi}(x)u_{\chi\downarrow}(y)
\Big) + 0,
\end{eqnarray*}
and
\begin{displaymath}
\dfrac{1}{2}\Big(\dfrac{\partial}{\partial x}\Big\vert_{x=y}
-\dfrac{\partial}{\partial y}\Big\vert_{y=x}\Big)
\Big(
2u_{\chi\downarrow}(x)
\psi_{\chi}(x)u_{\chi\downarrow}(y)
\Big) =1.
\end{displaymath}

\begin{lemma}
\label{Lem Mom tilde Dyson}
Let $(\tilde{\lambda}(x))_{x\geq 0}$ be the solution to
\eqref{Eq b Dyson tilde} with $\tilde{\lambda}(0)=0$.
Let $\nu$ be a finite family of positive integers,
with $\vert\nu\vert$ even.
Let $x_{1}\leq x_{2}\leq\dots\leq x_{m(\nu)}\in\R_{+}$.
Then,
\begin{displaymath}
\Big\langle \prod_{k=1}^{m(\nu)}
p_{\nu_{k}}(\tilde{\lambda}(x_{k}))
\Big\rangle_{\beta,n}^{\R_{+}} = 
P_{\nu}((\Y_{kk}=G_{\R_{+},\chi}(x_{k},x_{k}))_{1\leq k\leq m(\nu)},
(\cY_{k-1\,k}=
u_{\chi\downarrow}(x_{k})/u_{\chi\downarrow}(x_{k-1}))_{2\leq k\leq m(\nu)}).
\end{displaymath}
\end{lemma}

\begin{proof}
From Lemma \ref{Lem scale change} and Proposition 
\ref{Prop Mom b Dyson} it follows that
\begin{eqnarray*}
\Big\langle \prod_{k=1}^{m(\nu)}
p_{\nu_{k}}(\tilde{\lambda}(x_{k}))
\Big\rangle_{\beta,n}^{\R_{+}}
&=&
\Big(\prod_{k=1}^{m(\nu)}u_{\chi\downarrow}(x_{k})^{\nu_{k}}\Big)
Q_{\nu}((\Y_{kk}=2\psi_{\chi}(x_{k}))_{1\leq k\leq m(\nu)}).
\\&=&
\Big(\prod_{k=1}^{m(\nu)}u_{\chi\downarrow}(x_{k})^{\nu_{k}}\Big)
P_{\nu}((\Y_{kk}=2\psi_{\chi}(x_{k}))_{1\leq k\leq m(\nu)},
(\cY_{k-1\,k}=1)_{2\leq k\leq m(\nu)}).
\end{eqnarray*}
Further, let $A$ be a monomial of $P_{\nu}$. 
One has to check that
\begin{multline*}
\Big(\prod_{k=1}^{m(\nu)}u_{\chi\downarrow}(x_{k})^{\nu_{k}}\Big)
A((\Y_{kk}=2\psi_{\chi}(x_{k}))_{1\leq k\leq m(\nu)},
(\cY_{k-1\,k}=1)_{2\leq k\leq m(\nu)})
\\=
A((\Y_{kk}=2\psi_{\chi}(x_{k})u_{\chi\downarrow}(x_{k})^{2})_{1\leq k\leq m(\nu)},
(\cY_{k-1\,k}=
u_{\chi\downarrow}(x_{k})/u_{\chi\downarrow}(x_{k-1}))_{2\leq k\leq m(\nu)}).
\end{multline*}
This amounts to counting the power for each 
$u_{\chi\downarrow}(x_{k})$ on both sides.
On the left-hand side, each 
$u_{\chi\downarrow}(x_{k})$ appears with power $\nu_{k}$.
The power of $u_{\chi\downarrow}(x_{k})$
on the right-hand side is
\begin{displaymath}
2\deg_{\Y_{kk}}A
+ \deg_{\Y_{k-1\,k}}A - \deg_{\Y_{k\,k+1}}A .
\end{displaymath}
By \eqref{Eq deg A}, this is again $\nu_{k}$.
Finally, by \eqref{Eq Green psi},
\begin{displaymath}
2\psi_{\chi}(x_{k})u_{\chi\downarrow}(x_{k})^{2}
=G_{\R_{+},\chi}(x_{k},x_{k}).
\qedhere
\end{displaymath}
\end{proof}

\begin{proof}[Proof of Proposition \ref{Prop Iso b Dyson}]
It is enough to show \eqref{Eq Iso b Dyson} for functionals $F$
of form
\begin{displaymath}
F((\ell(x))_{x\geq 0})
=\exp\Big(-\int_{\R_{+}}\ell(x)\chi(x) dx\Big),
\end{displaymath}
where $\chi$ is a continuous non-negative function with compact support in $(0,+\infty)$.
For such a $\chi$,
\begin{multline*}
\Big\langle
\prod_{k=1}^{m(\nu)}
p_{\nu_{k}}(\lambda(x_{k}))
\exp\Big(-\dfrac{1}{2}\int_{\R_{+}}p_{2}(\lambda(x))
\chi(x) dx\Big)
\Big\rangle_{\beta,n}^{\R_{+}}
=
\\
\Big\langle
\exp\Big(-\dfrac{1}{2}\int_{\R_{+}}p_{2}(\lambda(x))
\chi(x) dx\Big)
\Big\rangle_{\beta,n}^{\R_{+}}
\Big\langle
\prod_{k=1}^{m(\nu)}
p_{\nu_{k}}(\tilde{\lambda}(x_{k}))
\Big\rangle_{\beta,n}^{\R_{+}},
\end{multline*}
where $\tilde{\lambda}$ is given by
\eqref{Eq b Dyson tilde}, with
$\tilde{\lambda}(0)=0$.
The symmetric moments of $\tilde{\lambda}$ are given by
Lemma \ref{Lem Mom tilde Dyson}.
To conclude, we use that
\begin{displaymath}
\int_{\gamma}\exp\Big(-\int_{\R_{+}}L^{z}(\gamma)\chi(z) dz\Big)
\mu_{\R_{+}}^{x,x}(d\gamma)
=G_{\R_{+},\chi}(x,x),
\end{displaymath}
and for $0<x<y$,
\begin{displaymath}
\int_{\gamma}\exp\Big(-\int_{\R_{+}}L^{z}(\gamma)\chi(z) dz\Big)
\check{\mu}^{x,y}(d\gamma)
=\dfrac{G_{\R_{+},\chi}(x,y)}{G_{\R_{+},\chi}(x,x)}
=\dfrac{u_{\chi\downarrow}(y)}{u_{\chi\downarrow}(x)};
\end{displaymath}
see \cite[Section~3.2]{Lupu20131dimLoops}.
\end{proof}

\subsection{The stationary case}

In this section we consider the stationary $\beta$-Dyson's Brownian motion on the whole line and state the analogues of
Propositions \ref{Prop b Dyson}, \ref{Prop Mom b Dyson}
and \ref{Prop Iso b Dyson} for it.
The proofs are omitted, as they are similar to the previous ones.
As previously, $n\geq 2$ and $\beta\geq 0$.
Let $K>0$.
We consider the process
$(\lambda(x)=
(\lambda_{1}(x),\dots,
\lambda_{n}(x)))_{x\in\R}$
with $\lambda_{1}(x)\geq\dots\geq\lambda_{n}(x)$,
satisfying the SDE
\begin{equation}
\label{Eq b Dyson stat}
d\lambda_{j}(x)=
\sqrt{2} dW_{j}(x)
-\sqrt{2K}\,\lambda_{j}(x)
+ \beta \sqrt{2K} \sum_{j'\neq j} 
\dfrac{dx}{\lambda_{j}(x)-\lambda_{j'}(x)},
\end{equation}
the $dW_{j}$, $1\leq j\leq n$, being $n$ i.i.d. white noises on $\R$,
and $\lambda$ being stationary, with
$(2K)^{\frac{1}{4}}\lambda(x)$ being distributed according to
\eqref{Eq GbE} (up to reordering of the $\lambda_{j}(x)$-s).

\begin{prop}
\label{Prop stat b Dyson}
The following holds.
\begin{enumerate}
\item The process
$\big(\frac{1}{\sqrt{n}}p_{1}(\lambda(x))\big)_{x\in \R}$
has the same law as $\phi_{K}$.
\item 
Consider a 1D Brownian loop soup $\calL^{\alpha}_{K}$,
with $\alpha$ given by \eqref{Eq alpha d b n}.
The process $(\frac{1}{2}p_{2}(\lambda(x)))_{x\in\R}$
has the same law as the occupation field
$(L^{x}(\calL^{\alpha}_{K}))_{x\in \R}$.
\item 
The processes
$(p_{1}(\lambda(x)))_{x\in \R}$
and 
$\big(\lambda(x)
-\frac{1}{n}p_{1}(\lambda(x))\big)_{x\in\R}$
are independent.
\item 
Let
$\calL^{\alpha-\frac{1}{2}}_{K}$ and
$\widetilde{\calL}^{\frac{1}{2}}_{K}$
be two independent 1D Brownian loop soups,
$\alpha$ given by \eqref{Eq alpha d b n}.
Then, one has the following identity in law between pairs of processes:
\begin{displaymath}
\Big(\frac{1}{2}\Big(p_{2}(\lambda(x))
-\frac{1}{n}p_{1}(\lambda(x))^{2}\Big),
\frac{1}{2n}p_{1}(\lambda(x))^{2}\Big)_{x\in \R}
\stackrel{\text{(law)}}{=}
(L^{x}(\calL^{\alpha-\frac{1}{2}}_{K}),
L^{x}(\widetilde{\calL}^{\frac{1}{2}}_{K}))_{x\in\R}.
\end{displaymath}
\end{enumerate}
\end{prop}

We will denote by $\langle\cdot\rangle^{K}_{\beta,n}$ the expectation with respect to the stationary $\beta$-Dyson's Brownian motion.
Given $\nu$ a finite family of positive integers with
$\vert\nu\vert$ even and 
$x_{1}< x_{2}<\dots<x_{m(\nu)}\in\R$,
$\mu_{K}^{\nu,x_{1},\dots,x_{m(\nu)}}(d\Upsilon)$
(also depending on $\beta$ and $n$)
will be the measure on finite families of
continuous paths obtained by substituting
in the polynomial $P_{\nu}=P_{\nu,\beta,n}$
for each variable $\Y_{kk}$ the measure
$\mu_{K}^{x_{k},x_{k}}$,
and for each variable $\cY_{k-1\,k}$
the measure $\check{\mu}_{K}^{x_{k-1},x_{k}}$.

\begin{prop}
\label{Prop Iso stat b Dyson}
Let $\nu$ a finite family of positive integers with
$\vert\nu\vert$ even.
Let $x_{1}\leq x_{2}\leq\dots\leq x_{m(\nu)}\in\R$.
Then,
\begin{multline*}
\Big\langle \prod_{k=1}^{m(\nu)}
p_{\nu_{k}}(\lambda(x_{k}))
\Big\rangle_{\beta,n}^{K} = 
P_{\nu}((\Y_{kk}=1/\sqrt{2K})_{1\leq k\leq m(\nu)},
(\cY_{k-1\,k}=e^{-\sqrt{2K}(x_{k}-x_{k-1})})_{2\leq k\leq m(\nu)})
\\
\\ =
P_{\nu}((\Y_{kk}=G_{K}(x_{k},x_{k}))_{1\leq k\leq m(\nu)},
(\cY_{k-1\,k}=
G_{K}(x_{k-1},x_{k})/G_{K}(x_{k-1},x_{k-1}))_{2\leq k\leq m(\nu)}
).
\end{multline*}
Further, let $F$ be a bounded measurable functional on
$\mathcal{C}(\R)$.
For $x_{1}< x_{2}<\dots < x_{m(\nu)}\in\R$,
\begin{displaymath}
\Big\langle
\prod_{k=1}^{m(\nu)}
p_{\nu_{k}}(\lambda(x_{k}))
F\Big(\dfrac{1}{2}p_{2}(\lambda)\Big)
\Big\rangle_{\beta,n}^{K}
=\int_{\Upsilon}
\Big\langle
F\Big(\dfrac{1}{2}p_{2}(\lambda)+L(\Upsilon)\Big)
\Big\rangle_{\beta,n}^{K}
\mu_{K}^{\nu,x_{1},\dots,x_{m(\nu)}}(d\Upsilon).
\end{displaymath}
\end{prop}

\section{The case of general electrical networks: a construction for $n=2$ and further questions}
\label{Sec Graph}

\subsection{Formal polynomials for $n=2$}

In this section $n=2$, and $\beta$ is arbitrary, 
considered as a formal parameter.
Note that $d(\beta,n=2)=\beta+2$.
In Section \ref{Sec mome} we introduced the formal commuting polynomial variables $(\Y_{kk})_{k\geq 1}$. Here we further consider the commuting variables $(\Y_{kk'})_{1\leq k<k'}$, and by convention set
$\Y_{kk'}=\Y_{k'k}$ for $k'<k$.
Given $\tilde{\nu}=(\tilde{\nu}_{1},\dots, \tilde{\nu}_{m})$
with $\tilde{\nu}_{k}\in\mathbb{N}$ (value $0$ allowed),
$\mathfrak{P}_{\tilde{\nu},\beta}$ will be the following
multivariate polynomial in the variables
$(\Y_{kk'})_{1\leq k\leq k'\leq m}$:
\begin{displaymath}
\mathfrak{P}_{\tilde{\nu},\beta}:=
\operatorname{Perm}_{\frac{\beta+1}{2}}
((\Y_{f(i)f(j)})
_{1\leq i,j\leq \tilde{\nu}_{1}+\dots +\tilde{\nu}_{m}}),
\end{displaymath}
where $f$ is a map 
$f: \llbracket 1,\tilde{\nu}_{1}+\dots +\tilde{\nu}_{m}\rrbracket
\rightarrow \llbracket 1,m\rrbracket$,
such that for every $k\in\llbracket 1,m\rrbracket$,
$\vert f^{-1}(k)\vert=\tilde{\nu}_{k}$.
Recall the expression of the $\alpha$-permanents
\eqref{Eq Perm}.
It is clear that $\mathfrak{P}_{\tilde{\nu},\beta}$ does not depend on the particular choice of $f$.
In case $\tilde{\nu}_{1}=\dots = \tilde{\nu}_{m}=0$,
by convention we set $\mathfrak{P}_{\tilde{\nu},\beta}=1$.
Given $\nu$ a finite family of positive integers with
$\vert\nu\vert$ even,
let $\mathbf{k}_{\nu} : \llbracket 1,\vert\nu\vert\rrbracket
\mapsto \llbracket 1,m(\nu)\rrbracket$
be the map given by \eqref{Eq k nu}.
Let $\I_{\nu}$  be the following set of subsets of 
$\llbracket 1,\vert\nu\vert\rrbracket$:
\begin{displaymath}
\I_{\nu}:=
\{
I\subseteq\llbracket 1,\vert\nu\vert\rrbracket
\vert\,\forall k\in \llbracket 1,m(\nu)\rrbracket,
\vert\mathbf{k}_{\nu}^{-1}(k)\setminus I\vert
\text{ is even }
\},
\end{displaymath}
where $\vert\cdot\vert$
denotes the cardinal.
Note that necessarily, for every $I\in\I_{\nu}$,
the cardinal $\vert I\vert$ is even.
Let $\widehat{P}_{\nu,\beta}$ be
the following multivariate polynomial in the variables
$(\Y_{kk'})_{1\leq k\leq k'\leq m(\nu)}$:
\begin{displaymath}
\widehat{P}_{\nu,\beta}:=
\sum_{I\in \I_{\nu}}
2^{m(\nu)-\vert I\vert /2}
\Big(
\sum_{\substack{(\{a_{i},b_{i}\})_{1\leq i\leq \vert I\vert/2}\\
\text{partition in pairs}
\\
\text{of } I
}} 
\prod_{i=1}^{\vert I\vert/2}
\Y_{\mathbf{k}_{\nu}(a_{i})\mathbf{k}_{\nu}(b_{i})}
\Big)
\mathfrak{P}
_{(\frac{1}{2}\vert\mathbf{k}_{\nu}^{-1}(k)\setminus I\vert)_{1\leq k\leq m(\nu)},\beta}.
\end{displaymath}
By construction,
for  every $A$ monomial of
$\widehat{P}_{\nu,\beta}$ and 
every $k\in\llbracket 1,m(\nu)\rrbracket$,
\begin{equation}
\label{Eq degrees}
2\deg_{\Y_{kk}}A+
\sum_{\substack{1\leq k'\leq m(\nu)\\k'\neq k}}
\deg_{\Y_{kk'}}A = \nu_{k}.
\end{equation}

\begin{prop}
\label{Prop P nu n = 2}
Let $\nu$ be finite family of positive integers with
$\vert\nu\vert$ even.
$P_{\nu,\beta,n=2}$ is obtained from $\widehat{P}_{\nu,\beta}$
by replacing each variable
$\Y_{kk'}$ with $1\leq k<k'\leq m(\nu)$ by
$\Y_{kk}\prod_{k+1\leq r\leq k'}\cY_{r-1\,r}$:
\begin{displaymath}
P_{\nu,\beta,n=2}=
\widehat{P}_{\nu,\beta}\big(
\big(\Y_{kk'}=\Y_{kk}\prod_{k+1\leq r\leq k'}\cY_{r-1\,r}\big)_{1\leq k<k'\leq m(\nu)}\big).
\end{displaymath}
\end{prop}

\begin{proof}
Let be
\begin{displaymath}
\widetilde{P}_{\nu,\beta}:=\widehat{P}_{\nu,\beta}\big(
\big(\Y_{kk'}=\Y_{kk}\prod_{k+1\leq r\leq k'}\cY_{r-1\,r}\big)_{1\leq k<k'\leq m(\nu)}\big).
\end{displaymath}
We want to show the equality 
$\widetilde{P}_{\nu,\beta}=P_{\nu,\beta,n=2}$.
Since a direct combinatorial proof would be a bit lengthy, we proceed differently.
Let $\beta\geq 0$ and let
$(\lambda(x)
=(\lambda_{1}(x),\lambda_{2}(x)))_{x\geq 0}$
be the $\beta$-Dyson's Brownian motion \eqref{Eq b Dyson}
in the case $n=2$.
We use its construction through \eqref{Eq b Dyson n=2 Bessel}.
We claim that
for $x_{1},x_{2},\dots,x_{m(\nu)}\in\R_{+}$,
\begin{displaymath}
\Big\langle \prod_{k=1}^{m(\nu)}
p_{\nu_{k}}(\lambda(x_{k}))
\Big\rangle_{\beta,n=2}^{\R_{+}} =
\widehat{P}_{\nu,\beta}\big(
\big(\Y_{kk'}=G_{\R_{+}}(x_{k-1},x_{k})\big)_{1\leq k\leq k'\leq m(\nu)}\big). 
\end{displaymath}
Indeed, in the expansion of
\begin{displaymath}
\Big(\widetilde{W}(x_{k})+\rho(x_{k})\Big)^{\nu_{k}}
+\Big(\widetilde{W}(x_{k})-\rho(x_{k})\Big)^{\nu_{k}}
\end{displaymath}
only enter the even powers of $\rho(x_{k})$, 
which is how $\I_{\nu}$ appears.
Then one uses that the square Bessel process $(\rho(x))_{x\geq 0}$
is a $(\beta+1)/2$-permanental field with kernel
$(G_{\R_{+}}(x,y))_{x,y\in\R_{+}}$.
Because of the particular form of $G_{\R_{+}}$, 
we have that
for $x_{1} \leq x_{2}\leq\dots\leq x_{m(\nu)}\in\R_{+}$,
\begin{displaymath}
\Big\langle \prod_{k=1}^{m(\nu)}
p_{\nu_{k}}(\lambda(x_{k}))
\Big\rangle_{\beta,n=2}^{\R_{+}} =
\widetilde{P}_{\nu,\beta}((\Y_{kk}=2x_{k})_{1\leq k\leq m(\nu)},
(\cY_{k-1\,k}=1)_{2\leq k\leq m(\nu)}).
\end{displaymath}
By combining with Corollary \ref{Cor mome P nu},
we get that the following multivariate polynomials in the variables
$(\Y_{kk})_{1\leq k\leq m(\nu)}$ are equal for $\beta\geq 0$:
\begin{displaymath}
\widetilde{P}_{\nu,\beta}(
(\cY_{k-1\,k}=1)_{2\leq k\leq m(\nu)})
=
P_{\nu,\beta,n=2}((\cY_{k-1\,k}=1)_{2\leq k\leq m(\nu)}).
\end{displaymath}
Since the coefficients of both are polynomials in $\beta$, 
the equality above holds for general $\beta$.
To conclude the equality 
$\widetilde{P}_{\nu,\beta}=P_{\nu,\beta,n=2}$,
we have to deal with the variables 
$(\cY_{k-1\,k})_{2\leq k\leq m(\nu)}$.
For this we use that both in case of $P_{\nu,\beta,n=2}$
and in case of $\widetilde{P}_{\nu,\beta}$, 
each monomial satisfies \eqref{Eq deg A}.
For $\widetilde{P}_{\nu,\beta}$ this follows from
\eqref{Eq degrees}.
\end{proof}

\subsection{A construction on discrete electrical networks for $n=2$}
\label{Sec a construction}

Let $\G=(V,E)$ be an undirected connected graph, 
with $V$ finite. 
We do not allow multiple edges or self-loops. 
The edges $\{x,y\}\in E$ are endowed with conductances $C(x,y)=C(y,x)>0$. 
There is also a non-uniformly zero killing measure
$(K(x))_{x\in V}$, with 
$K(x)\geq 0$.
We see $\G$ as an electrical network. 
Let $\Delta_{\G}$ denote the discrete Laplacian
\begin{displaymath}
(\Delta_{\G} f)(x)=
\sum_{y\sim x} C(x,y)(f(y)-f(x))
.
\end{displaymath}
Let $(G_{\G,K}(x,y))_{x,y\in V}$ be the massive Green's function
$G_{\G,K}=(-\Delta_{\G} + K)^{-1}$.
The (massive) real scalar Gaussian free field (GFF)
is the centered random Gaussian field on $V$ with
covariance $G_{\G,K}$, or equivalently with density
\begin{equation}
\label{EqDensityGFF}
\dfrac{1}{((2\pi)^{\vert V\vert} \det G_{\G,K})^{\frac{1}{2}}}
\exp\Big(
-\dfrac{1}{2}\sum_{x\in V}K(x)\varphi(x)^{2}
-\dfrac{1}{2}\sum_{\{x,y\}\in E}C(x,y)(\varphi(y)-\varphi(x))^{2}
\Big).
\end{equation}

Let $X_{t}$ be the continuous time Markov jump process to nearest neighbors with jump rates given by the conductances. 
The process $X_{t}$ is also killed by $K$. 
Let $\zeta\in (0,+\infty]$ be 
the first time $X_{t}$ gets killed by $K$. 
Let $p_{\G,K}(t,x,y)$ be the transition probabilities of
$(X_{t})_{0\leq t <\zeta}$. Then
$p_{\G,K}(t,x,y)=p_{\G,K}(t,y,x)$ and
\begin{displaymath}
G_{\G,K}(x,y)=\int_{0}^{+\infty}p_{\G,K}(t,x,y) dt.
\end{displaymath}
Let $\mathbb{P}_{\G,K}^{t,x,y}$ be the bridge probability measure from $x$ to $y$, where one conditions on $t<\zeta$. 
For $x,y\in V$, let
$\mu^{x,y}_{\G,K}$ be the following measure on paths:
\begin{displaymath}
\mu^{x,y}_{\G,K}(\cdot):=
\int_{0}^{+\infty}\mathbb{P}_{\G,K}^{t,x,y}(\cdot)p_{\G,K}(t,x,y) dt.
\end{displaymath}
It is the analogue of \eqref{Eq paths 1D}.
The total mass of $\mu^{x,y}_{\G,K}$ is
$G_{\G,K}(x,y)$, and the image of $\mu^{x,y}_{\G,K}$
by time reversal is $\mu^{y,x}_{\G,K}$.
Similarly, one defines the measure on (rooted) loops by
\begin{displaymath}
\mu_{\G,K}^{\rm loop}(d\gamma):=
\dfrac{1}{T(\gamma)}
\sum_{x\in V}
\mu^{x,x}_{\G,K}(d\gamma),
\end{displaymath}
where $T(\gamma)$ denotes the duration of the loop
$\gamma$.
It is the analogue of \eqref{Eq loop 1D}.
The measure $\mu_{\G,K}^{\rm loop}$ has an infinite total mass because
it puts an infinite mass on trivial "loops" that stay in one vertex.
For $\alpha>0$, one considers Poisson point processes
$\calL^{\alpha}_{\G,K}$ of intensity $\alpha\mu_{\G,K}^{\rm loop}$.
These are (continuous time) \textit{random walk loop soups}.
For details, see
\cite{LawlerFerreras07RWLoopSoup,
LawlerLimic2010RW,LeJan2010LoopsRenorm,LeJan2011Loops}.

For a continuous time path $\gamma$ on $\G$ of duration $T(\gamma)$
and $x\in V$, we denote
\begin{displaymath}
L^{x}(\gamma):=\int_{0}^{T(\gamma)}\1_{\gamma(s)=x} ds.
\end{displaymath}
Further,
\begin{displaymath}
L^{x}(\calL^{\alpha}_{\G,K}):=
\sum_{\gamma\in\calL^{\alpha}_{\G,K}}L^{x}(\gamma).
\end{displaymath}
One has equality in law between
$(L^{x}(\calL^{\frac{1}{2}}_{\G,K}))_{x\in V}$
and $(\frac{1}{2}\phi_{\G,K}(x)^{2})_{x\in V}$,
where $\phi_{\G,K}$ is the GFF distributed according to
\eqref{EqDensityGFF}
\cite{LeJan2010LoopsRenorm,LeJan2011Loops}.
This is the analogue of \eqref{Iso Le Jan 1D}.
For general $\alpha>0$, the occupation field
$(L^{x}(\calL^{\alpha}_{\G,K}))_{x\in V}$
is the $\alpha$-permanental field with kernel
$G_{\G,K}$ \cite{LeJan2010LoopsRenorm,LeJan2011Loops,LeJanMarcusRosen12Loops}.
In this sense it is analogous to squared Bessel processes.
If $(\chi(x))_{x\in V}\in \R^{V}$ is such that
$-\Delta_{\G} + K - \chi$ is positive definite, then
\begin{equation}
\label{Eq Laplace alpha Perm}
\E\Big[\exp\Big(\sum_{x\in V} \chi(x)L^{x}(\calL^{\alpha}_{\G,K})\Big)\Big]
=
\bigg(
\dfrac{\det ( -\Delta_{\G} + K )}{\det ( -\Delta_{\G} + K - \chi )}
\bigg)^{\alpha}.
\end{equation}
See Corollary 5 in \cite{LeJan2010LoopsRenorm}
and Corollary 1, Section 4.1 in \cite{LeJan2011Loops}.

Now we proceed with our construction.
Fix $\beta>-1$.
Let $\alpha=\frac{1}{2}d(\beta,n=2)=\frac{\beta+2}{2}>\frac{1}{2}$.
Let $\phi_{\G,K}$ be a GFF distributed according to
\eqref{EqDensityGFF}, and
$\calL^{\alpha-\frac{1}{2}}_{\G,K}$ an 
independent random walk loop soup. For $x\in V$ we set
\begin{displaymath}
\lambda_{1}(x):=\dfrac{1}{\sqrt{2}}\phi_{\G,K}(x)
+\sqrt{L^{x}(\calL^{\alpha-\frac{1}{2}}_{\G,K})},
\qquad
\lambda_{2}(x):=\dfrac{1}{\sqrt{2}}\phi_{\G,K}(x)
-\sqrt{L^{x}(\calL^{\alpha-\frac{1}{2}}_{\G,K})},
\end{displaymath}
and 
$\lambda:=(\lambda_{1}(x),\lambda_{2}(x))_{x\in V}$.
$\langle\cdot\rangle_{\beta,n=2}^{\G,K}$ will denote the
expectation with respect to $\lambda$.
As in Section \ref{Sec Dynkin Dyson},
$\Upsilon=(\gamma_{1},\dots,\gamma_{J(\Upsilon)})$
will denote a generic family of continuous time paths,
this time on the graph $\G$.
For $x\in V$,
\begin{displaymath}
L^{x}(\Upsilon):=
\sum_{i=1}^{J(\Upsilon)}L^{x}(\gamma_{i}),
\end{displaymath}
and $L(\Upsilon)$ will denote the occupation field of 
$\Upsilon$, $x\mapsto L^{x}(\Upsilon)$.
Given $\nu$ a finite family of positive integers with
$\vert\nu\vert$ even, and
$x_{1},x_{2},\dots,x_{m(\nu)}\in V$,
$\hat{\mu}_{\G,K}^{\nu,\beta,x_{1},\dots,x_{m(\nu)}}$
will denote the measure on families of
$\vert\nu\vert/2$ paths on $\G$
obtained by substituting in the polynomial
$\widehat{P}_{\nu,\beta}$
for each variable $\Y_{kk'}$,
$1\leq k\leq k'\leq m(\nu)$,
the measure $\mu_{\G,K}^{x_{k},x_{k'}}$.
The order of the paths will not matter.

\begin{prop}
The following holds.
\begin{enumerate}
\item For every $x\in V$,
$(\lambda_{1}(x)/\sqrt{G_{\G,K}(x,x)},
\lambda_{2}(x)/\sqrt{G_{\G,K}(x,x)})$
is distributed, up to reordering,
according to \eqref{Eq GbE} for $n=2$.
\item Let $x,y\in V$. Let
\begin{equation}
\label{Eq eta}
\eta=\dfrac{G_{\G,K}(x,x)G_{\G,K}(y,y)}
{G_{\G,K}(x,y)^{2}}\geq 1.
\end{equation}
Then the couple 
$(\sqrt{2}\lambda(x)/\sqrt{G_{\G,K}(x,x)},
\sqrt{2\eta}\lambda(y)/\sqrt{G_{\G,K}(y,y)})$
is distributed like the
$\beta$-Dyson's Brownian motion
\eqref{Eq b Dyson} at points $1$ and $\eta$,
for $n=2$.
\item Let $\nu$ be finite family of positive integers with
$\vert\nu\vert$ even and $x_{1},x_{2},\dots,x_{m(\nu)}\in V$.
Then
\begin{displaymath}
\Big\langle\prod_{k=1}^{m(\nu)}
p_{\nu_{k}}(\lambda(x_{k}))
\Big\rangle^{\G,K}_{\beta,n=2}
=\widehat{P}_{\nu,\beta}
((\Y_{kk'}=G_{\G,K}(x_{k},x_{k'}))_{1\leq k\leq k'\leq m(\nu)}).
\end{displaymath}
\item (BFS-Dynkin's isomorphism)
Moreover, given $F$ a measurable bounded function on
$\R^{V}$,
\begin{equation}
\label{Eq Dynkin graph n = 2}
\Big\langle\prod_{k=1}^{m(\nu)}
p_{\nu_{k}}(\lambda(x_{k}))
F\Big(\dfrac{1}{2}p_{2}(\lambda)\Big)
\Big\rangle^{\G,K}_{\beta,n=2}
=
\int_{\Upsilon}
\Big\langle
F\Big(\dfrac{1}{2}p_{2}(\lambda)
+L(\Upsilon)\Big)
\Big\rangle^{\G,K}_{\beta,n=2}
\hat{\mu}_{\G,K}^{\nu,\beta,x_{1},\dots,x_{m(\nu)}}
(d\Upsilon).
\end{equation}
\item For $\beta\in\{1,2,4\}$,
$(\lambda_{1}(x),\lambda_{2}(x))_{x\in V}$
is distributed like the ordered family of eigenvalues in a
GFF with values in $2\times 2$
real symmetric $(\beta=1)$, complex Hermitian $(\beta=2)$,
resp. quaternionic Hermitian $(\beta=4)$ matrices,
with density proportional to
\begin{equation}
\label{Eq matrix GFF}
\exp\Big(
-\dfrac{1}{2}\sum_{x\in V}K(x)\Tr(M(x)^{2})
-\dfrac{1}{2}\sum_{\{x,y\}\in E}
C(x,y)\Tr((M(y)-M(x))^{2})
\Big).
\end{equation}
\item Assume that $\beta>0$.
Let $\phi_{1}$ and $\phi_{2}$ be two independent scalar GFFs
distributed according to \eqref{EqDensityGFF}.
$\calL^{\alpha-1}_{\G,K}$ be a random walk loop soup independent from
$(\phi_{1},\phi_{2})$, with still $\alpha=\frac{\beta+2}{2}$.
Then $(\lambda_{1}(x),\lambda_{2}(x))_{x\in V}$
is distributed as the ordered family of eigenvalues 
in the matrix-valued field
\begin{equation}
\label{Eq another matrix}
\left(
\begin{array}{cc}
\phi_{1}(x) & \sqrt{L^{x}(\calL^{\alpha-1}_{\G,K})} \\ 
\sqrt{L^{x}(\calL^{\alpha-1}_{\G,K})} & \phi_{2}(x)
\end{array} 
\right),~~x\in V.
\end{equation}
\item Given another killing measure 
$\widetilde{K}\in\R_{+}^{V}$, non uniformly zero,
and $\tilde{\lambda}=(\tilde{\lambda}_{1},\tilde{\lambda}_{2})$
the field obtained by using $\widetilde{K}$ instead of $K$,
the density of the law of $\tilde{\lambda}$
with respect to that of $\lambda$ is
\begin{displaymath}
\bigg(
\dfrac{\det ( -\Delta_{\G} + \widetilde{K} )}
{\det ( -\Delta_{\G} + K)}
\bigg)^{\frac{\beta+2}{2}}
\exp\Big(-\dfrac{1}{2}\sum_{x\in V}
(\widetilde{K}(x)-K(x))p_{2}(\lambda(x))\Big).
\end{displaymath}
\end{enumerate}
\end{prop}

\begin{proof}
(1) This follows from Proposition \ref{Proberty GbE} and the fact that 
$\phi_{\G,K}(x)/\sqrt{G_{\G,K}(x,x)}$
is distributed according to $\mathcal{N}(0,1)$,
and 
$L^{x}(\calL^{\alpha-\frac{1}{2}}_{\G,K})/\sqrt{G_{\G,K}(x,x)}$
according to
$\operatorname{Gamma}\big(\alpha-\frac{1}{2},1\big)$.

(2) One uses the decomposition \eqref{Eq b Dyson n=2 Bessel}
of a $\beta$-Dyson's Brownian motion for $n=2$.
Indeed,
$(\sqrt{2}\phi_{\G,K}(x)/\sqrt{G_{\G,K}(x,x)},
\sqrt{2\eta}\phi_{\G,K}(y)/\sqrt{G_{\G,K}(y,y)})$
and $(\phi_{\R_{+}}(1),\phi_{\R_{+}}(\eta))$
are two Gaussian vectors with the same distribution, with covariance matrix given by
\begin{equation}
\label{Eq cov matrix}
\left(
\begin{array}{cc}
2 & 2 \\ 
2 & 2\eta
\end{array} 
\right)
.
\end{equation}
Moreover, $(\sqrt{2}L^{x}(\calL^{\alpha-\frac{1}{2}}_{\G,K})/\sqrt{G_{\G,K}(x,x)},
\sqrt{2\eta}L^{y}(\calL^{\alpha-\frac{1}{2}}_{\G,K})/\sqrt{G_{\G,K}(y,y)})$
is distributed as $(\rho(1),\rho(\eta))$,
a two-dimensional marginal of a Bessel process
of dimension $\beta+1$.
The latter can be seen using the moments, that characterize the 
finite-dimensional marginals of the Bessel process $\rho$.
In both cases those are $(\beta+1)/2$-permanents,
with coefficients given by the matrix \eqref{Eq cov matrix}.

(3) This follows by expanding
\begin{equation}
\label{Eq to expand}
\Big(\dfrac{1}{\sqrt{2}}\phi_{\G,K}(x_{k})
+\sqrt{L^{x_{k}}(\calL^{\alpha-\frac{1}{2}}_{\G,K})}\Big)^{\nu_{k}} + \Big(
\dfrac{1}{\sqrt{2}}\phi_{\G,K}(x_{k})
-\sqrt{L^{x_{k}}(\calL^{\alpha-\frac{1}{2}}_{\G,K})}
\Big)^{\nu_{k}}
\end{equation}
for every $k\in\llbracket 1, m(\nu)\rrbracket$.
In this decomposition only the integer powers of
$L^{x_{k}}(\calL^{\alpha-\frac{1}{2}}_{\G,K})$
survive cancellation.
The moments of $(\phi_{\G,K}(x_{k}))_{1\leq k\leq m(\nu)}$
give rise to the Wick part in $\widehat{P}_{\nu,\beta}$
(sums over partitions in pairs).
The moments of 
$(L^{x_{k}}(\calL^{\alpha-\frac{1}{2}}_{\G,K}))_{1\leq k\leq m(\nu)}$
give rise to the permanental part in $\widehat{P}_{\nu,\beta}$.

(4) The GFF $\phi_{\G,K}$ satisfies the BFS-Dynkin isomorphism;
see \cite[Theorem~2.2]{BFS82Loop}, 
\cite[Theorems~6.1,~6.2]{Dynkin1984Isomorphism},
and \cite[Theorem~1]{Dynkin1984IsomorphismPresentation}.
Moreover, there is a version of 
BFS-Dynkin isomorphism for the occupation field
$L(\calL^{\alpha-\frac{1}{2}}_{\G,K})$
obtained by applying Palm's identity to Poisson point processes;
see \cite[Theorem~1.3]{LeJanMarcusRosen12Loops} 
and \cite[Sections~3.4,~4.3]{Lupu20131dimLoops}.
More precisely,
for any $y_{1},\dots,y_{r}\in V$,
\begin{multline*}
\E\Big[
\prod_{i=1}^{r}
L^{y_{i}}(\calL^{\alpha-\frac{1}{2}}_{\G,K})
F(L(\calL^{\alpha-\frac{1}{2}}_{\G,K}))
\Big] =
\\
\sum_{\substack{\sigma \text{ permutation}
\\ \text{of } \{1,2,\dots ,r\}}}
\Big(\alpha-\dfrac{1}{2}\Big)^{\# \text{ cycles of } \sigma}
\int\displaylimits_{\gamma_{1},\dots ,\gamma_{r}}
\E\Big[
F(L(\calL^{\alpha-\frac{1}{2}}_{\G,K})
+L(\gamma_{1})+\dots + L(\gamma_{r}))
\Big]
\prod_{i=1}^{r}
\mu^{y_{i},y_{\sigma(i)}}_{\G,K}
(d\gamma_{i}).
\end{multline*}
Further, by expanding \eqref{Eq to expand}
for $k\in\llbracket 1, m(\nu)\rrbracket$,
we get that
$\prod_{k=1}^{m(\nu)}p_{\nu_{k}}(\lambda(x_{k}))$
is actually a polynomial in the variables
$(\phi_{\G,K}(x_{k}))_{1\leq k\leq m(\nu)}$
and
$(L^{x_{k}}(\calL^{\alpha-\frac{1}{2}}_{\G,K}))_{1\leq k\leq m(\nu)}$,
the non-integer powers of 
$L^{x_{k}}(\calL^{\alpha-\frac{1}{2}}_{\G,K})$
cancelling out.
Moreover,
\begin{displaymath}
\dfrac{1}{2}p_{2}(\lambda)
=
\dfrac{1}{2}\phi_{\G,K}^{2} + L(\calL^{\alpha-\frac{1}{2}}_{\G,K}).
\end{displaymath}
Since the fields $\phi_{\G,K}$ and 
$L(\calL^{\alpha-\frac{1}{2}}_{\G,K})$
are independent, on gets
\eqref{Eq Dynkin graph n = 2}
by combining the BFS-Dynkin isomorphism for 
$\phi_{\G,K}$ and 
the BFS-Dynkin isomorphism for $L(\calL^{\alpha-\frac{1}{2}}_{\G,K})$.

(5) Recall that for all three matrix spaces considered, $\beta+2$ is the dimension.
Given $(M(x))_{x\in V}$ a matrix field distributed according to
\eqref{Eq matrix GFF}, $M_{0}(x)$ will denote
$M(x)-\frac{1}{2}\Tr(M(x)) \mathbf{I}_{2}$, where $\mathbf{I}_{2}$ is the $2\times 2$ identity matrix, so that $\Tr(M_{0}(x))=0$.
Since the hyperplane of zero trace matrices is orthogonal to 
$\mathbf{I}_{2}$ for the inner product
$(A,B)\mapsto \Ren(\Tr(AB))$, we get that
$(M_{0}(x))_{x\in V}$ and
$(\Tr(M(x)))_{x\in V}$ are independent.
Moreover, $(\frac{1}{\sqrt{2}}\Tr(M(x)))_{x\in V}$
is distributed as the scalar GFF \eqref{EqDensityGFF}.
As for $(\Tr(M(x)^{2}))_{x\in V}$, on one hand it is the sum of
$\beta+2$ i.i.d. squares of scalar GFFs \eqref{EqDensityGFF}
corresponding to the entries of the matrices.
On the other hand,
\begin{displaymath}
\Tr(M(x)^{2})=\Tr(M_{0}(x)^{2})+\dfrac{1}{2}\Tr(M(x))^{2}.
\end{displaymath}
So $(\Tr(M_{0}(x)^{2}))_{x\in V}$ is distributed as the sum of
$\beta+1$ i.i.d. squares of scalar GFFs \eqref{EqDensityGFF}.
So in particular, this is the same distributions as for
$(2L^{x}(\calL^{\frac{\beta+1}{2}}_{\G,K}))_{x\in V}$.
Finally, the eigenvalues of $M(x)$ are
\begin{displaymath}
\dfrac{1}{2}\Tr(M(x)) \pm 
\dfrac{1}{\sqrt{2}}\sqrt{\Tr(M_{0}(x)^{2})}.
\end{displaymath}

(6) The eigenvalues of the matrix \eqref{Eq another matrix}
are
\begin{displaymath}
\dfrac{\phi_{1}(x)+\phi_{2}(x)}{2}
\pm
\sqrt{L^{x}(\calL^{\alpha-1}_{\G,K})+(\phi_{2}(x)-\phi_{1}(x))^{2}/4}.
\end{displaymath}
$(\phi_{1}+\phi_{2})/\sqrt{2}$
and 
$(\phi_{2}-\phi_{1})/\sqrt{2}$
are two independent scalar GFFs.
Moreover, 
\begin{displaymath}
L(\calL^{\alpha-1}_{\G,K})+\frac{1}{4}(\phi_{2}-\phi_{1})^{2}
\end{displaymath}
has same distribution as
$L(\calL^{\alpha-\frac{1}{2}}_{\G,K})$. 

(7) The density of the GFF $\phi_{\G,\widetilde{K}}$ with respect to
$\phi_{\G,K}$ is
\begin{displaymath}
\bigg(
\dfrac{\det ( -\Delta_{\G} + \widetilde{K} )}
{\det ( -\Delta_{\G} + K)}
\bigg)^{\frac{1}{2}}
\exp\Big(-\dfrac{1}{2}\sum_{x\in V}
(\widetilde{K}(x)-K(x))\varphi(x)^{2}\Big).
\end{displaymath}
The density of $L(\calL^{\alpha-\frac{1}{2}}_{\G,\widetilde{K}})$
with respect to $L(\calL^{\alpha-\frac{1}{2}}_{\G,K})$ is
\begin{displaymath}
\bigg(
\dfrac{\det ( -\Delta_{\G} + \widetilde{K} )}
{\det ( -\Delta_{\G} + K)}
\bigg)^{\alpha-\frac{1}{2}}
\exp\Big(-\sum_{x\in V}
(\widetilde{K}(x)-K(x))L^{x}(\calL^{\alpha-\frac{1}{2}}_{\G,K})\Big),
\end{displaymath}
as can be seen from the Laplace transform \eqref{Eq Laplace alpha Perm}.
\end{proof}

\subsection{Further questions}

Here we present our questions that motivated this paper.
The first question is combinatorial.
We would like to have the polynomials
$P_{\nu,\beta,n}$ given by
Definition \ref{Def Rec Pol} under a more explicit form.
The recurrence on polynomials \eqref{Eq rec poly}
is closely related to the
Schwinger-Dyson equation \eqref{Eq SD beta}.
Its very form suggests that the polynomials
$P_{\nu,\beta,n}$ might be expressible
as weighted sums over maps drawn on 2D compact
surfaces (not necessarily connected),
where the maps associated to $\nu$ have
$m(\nu)$ vertices with degrees given by
$\nu_{1},\nu_{2},\dots, \nu_{m(\nu)}$,
with powers of $n$ corresponding to the number of faces.
This is indeed the case for $\beta\in\{1,2,4\}$,
and this corresponds to the topological expansion of matrix integrals
\cite{BIPZ78PlanarDiagrams,IZ80PlanarApproximation,MulaseWaldron03GSE,
Lupu19TopoExp}.

\begin{quest}
\label{Q Comb}
Is there a more explicit expression for the polynomials
$P_{\nu,\beta,n}$? Can they be expressed as weighted sums over
the maps on 2D surfaces (topological expansion)?
\end{quest}

The second question is whether there is a natural generalization of
Gaussian beta ensembles and $\beta$-Dyson's Brownian motion to electrical networks. For $n=2$, such a generalization was given in
Section \ref{Sec a construction}.

\begin{quest}
\label{Q Gen}
We are in the setting of an electrical network
$\G=(V,E)$ endowed with a killing measure $K$,
as in Section \ref{Sec a construction}.
Given $n\geq 3$ and $\beta>-\frac{2}{n}$, is there a 
distribution on the fields
$(\lambda(x)
=(\lambda_{1}(x),\lambda_{2}(x),\dots,\lambda_{n}(x)))_{x\in V}$,
with $\lambda_{1}(x)>\lambda_{2}(x)>\dots>\lambda_{n}(x)$,
satisfying the following properties?
\begin{enumerate}
\item For $\beta\in\{1,2,4\}$, $\lambda$ is distributed as the fields of ordered eigenvalues in a GFF with values into 
$n\times n$ matrices, real symmetric $(\beta=1)$,
complex Hermitian $(\beta=2)$,
resp. quaternionic Hermitian $(\beta=4)$.
\item For $\beta=0$, $\lambda$ is obtained by reordering
$n$ i.i.d. scalar GFFs \eqref{EqDensityGFF}.
\item As $\beta\to -\frac{2}{n}$, $\lambda$ converges in law to
\begin{displaymath}
\Big(\dfrac{1}{\sqrt{n}}\phi_{\G,K},\dfrac{1}{\sqrt{n}}\phi_{\G,K},
\dots, \dfrac{1}{\sqrt{n}}\phi_{\G,K}\Big),
\end{displaymath}
where $\phi_{\G,K}$ is a scalar GFF \eqref{EqDensityGFF}.
\item For every $x\in V$,
$\lambda(x)/\sqrt{G_{\G,K}(x,x)}$ is distributed, up to reordering,
as the G$\beta$E \eqref{Eq GbE}.
\item For every $x,y\in V$, the couple
$(\sqrt{2}\lambda(x)/\sqrt{G_{\G,K}(x,x)},
\sqrt{2\eta}\lambda(y)/\sqrt{G_{\G,K}(y,y)})$,
with $\eta$ given by \eqref{Eq eta},
is distributed as the $\beta$-Dyson's Brownian motion
\eqref{Eq b Dyson} at points $1$ and $\eta$.
\item The fields $p_{1}(\lambda)$
and $\lambda-\frac{1}{n}p_{1}(\lambda)$
are independent.
\item The field $\frac{1}{\sqrt{n}}p_{1}(\lambda)$
is distributed as a scalar GFF \eqref{EqDensityGFF}.
\item The field $\frac{1}{2}\big(p_{2}(\lambda)
-\frac{1}{n}p_{1}(\lambda)^{2}\big)$ is
the $\alpha-\frac{1}{2}$-permanental field
with kernel $G_{\G,K}$,
where $\alpha=\frac{1}{2}d(\beta,n)$,
and in particular is distributed as the occupation field
of the continuous-time random walk loop soup
$\calL_{\G,K}^{\alpha-\frac{1}{2}}$.
\item The field 
$\frac{1}{2}p_{2}(\lambda)$ is
the $\alpha$-permanental field
with kernel $G_{\G,K}$,
where $\alpha=\frac{1}{2}d(\beta,n)$,
and in particular is distributed as the occupation field
of the continuous-time random walk loop soup
$\calL_{\G,K}^{\alpha}$
(already implied by (6)+(7)+(8)).
\item The symmetric moments 
\begin{displaymath}
\Big\langle\prod_{k=1}^{m(\nu)}p_{\nu_{k}}(\lambda(x_{k}))
\Big\rangle_{\beta,n}^{\G,K}
\end{displaymath}
are linear combination of products
\begin{displaymath}
\prod_{1\leq k\leq k'\leq m(\nu)}G_{\G,K}(x_{k},x_{k'})^{a_{kk'}},
\end{displaymath}
with $a_{kk'}\in\N$
and for every $k\in\llbracket 1,m(\nu)\rrbracket$,
\begin{displaymath}
2a_{kk}+\sum_{\substack{1\leq k'\leq m(\nu)\\k'\neq k}} a_{kk'} 
=\nu_{k},
\end{displaymath}
the coefficients of the linear combination being
universal polynomials in $\beta$ and $n$,
not depending on the electrical network and its parameters;
see also Question \ref{Q Comb}.
\item Given
$\widetilde{K}\in\R_{+}^{V}$, non-uniformly zero,
and $\tilde{\lambda}=(\tilde{\lambda}_{1},\tilde{\lambda}_{2},
\dots,\tilde{\lambda}_{n})$
the field associated to the killing measure
$\widetilde{K}$ instead of $K$,
the law of $\tilde{\lambda}$ has the following density with respect
to that of $\lambda$:
\begin{displaymath}
\bigg(
\dfrac{\det ( -\Delta_{\G} + \widetilde{K} )}
{\det ( -\Delta_{\G} + K)}
\bigg)^{\frac{1}{2}d(\beta,n)}
\exp\Big(-\dfrac{1}{2}\sum_{x\in V}
(\widetilde{K}(x)-K(x))p_{2}(\lambda(x))\Big).
\end{displaymath}
\item $\lambda$ satisfies a BFS-Dynkin type isomorphism
with continuous time random walks
(already implied by (10)+(11)).
\end{enumerate}
\end{quest}

If the graph $\G$ is a tree, the natural generalization $\lambda$ of the $\beta$-Dyson's Brownian motion is straightforward to construct, 
at least for $\beta\geq 0$. 
In absence of cycles, $\lambda$ satisfies a Markov property,
and along each branch of the tree one has the values of a
$\beta$-Dyson's Brownian motion at different positions.
On the random walk loop soup side, (8) and (9) is ensured by the covariance of the loop soups under the rewiring of graphs;
see \cite[Chapter~7]{LeJan2011Loops}. 
Constructing $\lambda$ on a tree for 
$\beta\in \big(-\frac{2}{n},0\big)$ is a matter of constructing
the corresponding $\beta$-Dyson's Brownian motion.
However, if the graph $\G$ contains cycles,
constructing $\lambda$ is not immediate, and we have not encountered such a construction in the literature.
One does not expect a Markov property, since already for
$\beta\in\{1,2,4\}$ one has to take into account the angular part of the matrices.

\newpage

\section*{Appendix: A list of moments for G$\beta$E and the corresponding formal polynomials}

\begin{eqnarray*}
\langle p_{1}(\lambda)^{2}\rangle_{\beta,n}
&=& n,
\\
P_{(1,1)}
&=& n \Y_{11} \cY_{12},
\end{eqnarray*}
\begin{eqnarray*}
\langle p_{2}(\lambda)\rangle_{\beta,n}
&=&\dfrac{\beta}{2} n^{2} + \Big(1-\dfrac{\beta}{2}\Big) n
=d(\beta,n),
\\
P_{(2)}
&=&\Big(\dfrac{\beta}{2} n^{2} +
 \Big(1-\dfrac{\beta}{2}\Big)n\Big) \Y_{11}
 = d(\beta,n)\Y_{11},
\end{eqnarray*}
\begin{eqnarray*}
\langle p_{1}(\lambda)^{4}\rangle_{\beta,n}&=& 3n^{2},
\\
P_{(1,1,1,1)}&=&
n^{2}\Y_{11}\cY_{12}\Y_{33}\cY_{34}+
2n^{2}\Y_{11}\cY_{12}\Y_{22}\cY_{23}^{2}\cY_{34},
\end{eqnarray*}
\begin{eqnarray*}
\langle p_{2}(\lambda)p_{1}(\lambda)^{2}\rangle_{\beta,n}&=&
\dfrac{\beta}{2} n^{3}
+\Big(1-\dfrac{\beta}{2}\Big) n^{2}
+ 2n,
\\
P_{(2,1,1)}&=&
\Big(\dfrac{\beta}{2} n^{3} + \Big(1-\dfrac{\beta}{2}\Big) n^{2}\Big)
\Y_{11}\Y_{22}\cY_{23}
+ 2n \Y_{11}^{2}\cY_{12}^{2}\cY_{23},
\\
P_{(1,2,1)}&=&
\Big(\dfrac{\beta}{2} n^{3} + \Big(1-\dfrac{\beta}{2}\Big) n^{2}
+2n\Big)
\Y_{11}\cY_{12}\Y_{22}\cY_{23},
\\
P_{(1,1,2)}&=&
\Big(\dfrac{\beta}{2} n^{3} + \Big(1-\dfrac{\beta}{2}\Big) n^{2}\Big)
\Y_{11}\cY_{12}\Y_{33}
+ 2n \Y_{11}\cY_{12}\Y_{22}\cY_{23}^{2},
\end{eqnarray*}
\begin{eqnarray*}
\langle p_{2}(\lambda)^{2}\rangle_{\beta,n}&=&
\dfrac{\beta^{2}}{4} n^{4}
+
2\dfrac{\beta}{2}\Big(1-\dfrac{\beta}{2}\Big)
 n^{3}
\\ && +
\Big(\Big(1-\dfrac{\beta}{2}\Big)^{2}
+2\dfrac{\beta}{2}\Big) n^{2}
+2\Big(1-\dfrac{\beta}{2}\Big) n
\\
&=& d(\beta,n)(d(\beta,n) +2),
\\
P_{(2,2)}&=&
\Big(
\dfrac{\beta^{2}}{4} n^{4}
+2\dfrac{\beta}{2}\Big(1-\dfrac{\beta}{2}\Big) n^{3}
+\Big(1-\dfrac{\beta}{2}\Big)^{2} n^{2}
\Big) \Y_{11}\Y_{22}
\\ && +
\Big(
2\dfrac{\beta}{2} n^{2}
+ 2\Big(1-\dfrac{\beta}{2}\Big) n
\Big) \Y_{11}^{2}\cY_{12}^{2},
\end{eqnarray*}
\begin{eqnarray*}
\langle p_{3}(\lambda)p_{1}(\lambda)\rangle_{\beta,n}&=&
3\dfrac{\beta}{2} n^{2}
+
3\Big(1-\dfrac{\beta}{2}\Big) n,
\\
P_{(3,1)} &=&
\Big(3\dfrac{\beta}{2} n^{2}
+
3\Big(1-\dfrac{\beta}{2}\Big) n 
\Big)
\Y_{11}^{2}\cY_{12},
\\
P_{(1,3)} &=&
\Big(3\dfrac{\beta}{2} n^{2}
+
3\Big(1-\dfrac{\beta}{2}\Big) n 
\Big)
\Y_{11}\cY_{12}\Y_{22},
\end{eqnarray*}
\begin{eqnarray*}
\langle p_{4}(\lambda)\rangle_{\beta,n}&=&
2\dfrac{\beta^{2}}{4} n^{3}
+
5 \dfrac{\beta}{2}\Big(1-\dfrac{\beta}{2}\Big)
n^{2}
+
\Big(
\dfrac{\beta}{2}+
3\Big(1-\dfrac{\beta}{2}\Big)^{2}
\Big)n,
\\
P_{(4)}&=&
\Big(
2\dfrac{\beta^{2}}{4} n^{3}
+ 
5 \dfrac{\beta}{2}\Big(1-\dfrac{\beta}{2}\Big)
n^{2}
+\Big(\dfrac{\beta}{2} 
+3\Big(1-\dfrac{\beta}{2}\Big)^{2}\Big)n
\Big) \Y_{11}^{2},
\end{eqnarray*}
\begin{eqnarray*}
\langle p_{3}(\lambda)^{2}\rangle_{\beta,n}&=&
12 \dfrac{\beta^{2}}{4}n^{3}
+
27 \dfrac{\beta}{2}\Big(1-\dfrac{\beta}{2}\Big)n^{2}
+
\Big(
3\dfrac{\beta}{2} +
15\Big(1-\dfrac{\beta}{2}\Big)^{2}
\Big) n,
\\
P_{(3,3)}&=&
9\Big(
\dfrac{\beta^{2}}{4}n^{3}
+2\dfrac{\beta}{2}\Big(1-\dfrac{\beta}{2}\Big)n^{2}
+\Big(1-\dfrac{\beta}{2}\Big)^{2}n
\Big)
\Y_{11}^{2}\cY_{12}\Y_{22}
\\
&&+
3\Big(
\dfrac{\beta^{2}}{4}n^{3}
+3\dfrac{\beta}{2}\Big(1-\dfrac{\beta}{2}\Big)n^{2}
+
\Big(
\dfrac{\beta}{2} +
2\Big(1-\dfrac{\beta}{2}\Big)^{2}
\Big) n
\Big)\Y_{11}^{3}\cY_{12}^{3}.
\end{eqnarray*}

\section*{Acknowledgements}

The author thanks Guillaume Chapuy and Jérémie Bouttier for
discussions and references on the beta ensembles.
The author thanks Yves Le Jan and Wendelin Werner for their feedback on the preliminary version of the article.

This work was supported by the French National Research Agency (ANR) grant within the project MALIN (ANR-16-CE93-0003).

\bibliographystyle{acm}
\bibliography{titusbibnew}

\begin{thebibliography}{10}

\bibitem{AndersonGuionnetZeitouni09RandomMatrix}
{\sc Anderson, G.~W., Guionnet, A., and Zeitouni, O.}
\newblock {\em An introduction to random matrices}, vol.~118 of {\em Cambridge
  studies in advanced mathematics}.
\newblock Cambridge University Press, 2009.

\bibitem{BFS82Loop}
{\sc Brydges, D., Fröhlich, J., and Spencer, T.}
\newblock The random walk representation of classical spin systems and
  correlation inequalities.
\newblock {\em Communications in Mathematical Physics 83}, 1 (1982), 123--150.

\bibitem{BIPZ78PlanarDiagrams}
{\sc Brézin, E., Itzykson, C., Parisi, G., and Zuber, J.-B.}
\newblock Planar diagrams.
\newblock {\em Communications in Mathematical Physics 59\/} (1978), 35--51.

\bibitem{Chan92DysonBM}
{\sc Chan, T.}
\newblock The {W}igner semi-circle law and eigenvalues of matrix-valued
  diffusions.
\newblock {\em Probability Theory and Related Fields 93\/} (1992), 249–272.

\bibitem{CepaLepingle97DysonBM}
{\sc Cépa, E., and Lépingle, D.}
\newblock Diffusing particles with electrostatic repulsion.
\newblock {\em Probability Theory and Related Fields 107\/} (1997), 429–449.

\bibitem{CepaLepingle07NoMultCol}
{\sc Cépa, E., and Lépingle, D.}
\newblock No multiple collisions for mutually repelling {B}rownian particles.
\newblock In {\em Séminaire de {P}robabilités {XL}\/} (2007),
  C.~Donati-Martin, M.~Émery, A.~Rouault, and C.~Stricker, Eds., vol.~1899 of
  {\em Lecture {N}otes in {M}athematics}, Springer, pp.~241--246.

\bibitem{DumitriuEdelman02TriDiagBeta}
{\sc Dumitriu, I., and Edelman, A.}
\newblock Matrix models for beta ensembles.
\newblock {\em Journal of Mathematical Physics 43}, 11 (2002), 5830–5847.

\bibitem{Dynkin1984Isomorphism}
{\sc Dynkin, E.}
\newblock Gaussian and non-{G}aussian random fields associated with {M}arkov
  processes.
\newblock {\em Journal of Functional Analysis 55\/} (1984), 344--376.

\bibitem{Dynkin1984IsomorphismPresentation}
{\sc Dynkin, E.}
\newblock Local times and quantum fields.
\newblock In {\em Seminar on Stochastic Processes, Gainesville 1983\/} (1984),
  vol.~7 of {\em Progress in Probability and Statistics}, Birkhauser,
  pp.~69--84.

\bibitem{Dyson62BM}
{\sc Dyson, F.~J.}
\newblock A {B}rownian-motion model for the eigenvalues of a random matrix.
\newblock {\em Journal of Mathematical Physics 3\/} (1962), 1191--1198.

\bibitem{EynardKimuraRibault18RandomMatrices}
{\sc Eynard, B., Kimura, T., and Ribault, S.}
\newblock Random matrices.
\newblock arXiv:1510.04430, 2018.

\bibitem{FitzsimmonsRosen14LoopSoup}
{\sc Fitzsimmons, P.~J., and Rosen, J.}
\newblock Markovian loop soups: permanental processes and isomor-phism
  theorems.
\newblock {\em Electronic Journal of Probability 19\/} (2014), 1–30.

\bibitem{OxfordHandbookRandomMatrixBeta}
{\sc Forrester, P.}
\newblock {\em The {O}xford Handbook of Random Matrix Theory}, 1st~ed.
\newblock {O}xford Handbooks. Oxford University Press, 2015, ch.~20 {B}eta
  ensembles, pp.~415--432.

\bibitem{IZ80PlanarApproximation}
{\sc Itzykson, C., and Zuber, J.-B.}
\newblock The planar approximation. {I}{I}.
\newblock {\em Journal of Mathematical Physics 21}, 3 (1980), 411--421.

\bibitem{LaCroix09PhD}
{\sc La~Croix, M.}
\newblock {\em The combinatorics of the {J}ack parameter and the genus series
  for topological maps}.
\newblock PhD thesis, University of Waterloo, 2009.

\bibitem{LaCroix13SlidesBeta}
{\sc La~Croix, M.}
\newblock $\beta$-{G}aussian ensembles and the non-orientability of polygonal
  glueings.
\newblock Slides available on author's webpage
  http://www.math.uwaterloo.ca/~malacroi/, 2013.

\bibitem{Lawler19Bessel}
{\sc Lawler, G.~F.}
\newblock Notes on the {B}essel process.
\newblock Notes available on author's webpage
  http://www.math.uchicago.edu/~lawler/bessel18new.pdf, 2019.

\bibitem{LawlerLimic2010RW}
{\sc Lawler, G.~F., and Limic, V.}
\newblock {\em Random walk: a modern introduction}, vol.~123 of {\em Cambridge
  studies in advanced mathematics}.
\newblock Cambridge University Press, 2010.

\bibitem{LawlerFerreras07RWLoopSoup}
{\sc Lawler, G.~F., and Trujillo-Ferreras, J.~A.}
\newblock Random walk loop soup.
\newblock {\em Transactions of American Mathematical Society 359}, 2 (2007),
  767--787.

\bibitem{LawlerWerner04BMLoopSoup}
{\sc Lawler, G.~F., and Werner, W.}
\newblock The {B}rownian loop-soup.
\newblock {\em Probability Theory and Related Fields 128\/} (2004), 565--588.

\bibitem{LeJan2010LoopsRenorm}
{\sc Le~Jan, Y.}
\newblock Markov loops and renormalization.
\newblock {\em The Annals of Probability 38}, 3 (2010), 1280--1319.

\bibitem{LeJan2011Loops}
{\sc Le~Jan, Y.}
\newblock Markov paths, loops and fields.
\newblock In {\em 2008 St-Flour summer school\/} (2011), vol.~2026 of {\em
  Lecture Notes in Mathematics}, Springer.

\bibitem{LeJanMarcusRosen12Loops}
{\sc Le~Jan, Y., Marcus, M.~B., and Rosen, J.}
\newblock Permanental fields, loop soups and continuous additive functionals.
\newblock {\em The Annals of Probability 43}, 1 (2015), 44--84.

\bibitem{Lupu20131dimLoops}
{\sc Lupu, T.}
\newblock {\em Poisson ensembles of loops of one-dimensional diffusions},
  vol.~158 of {\em Mémoires de la SMF}.
\newblock Société Mathématique de France, Paris, 2018.

\bibitem{Lupu19TopoExp}
{\sc Lupu, T.}
\newblock Topological expansion in isomorphism theorems between matrix-valued
  fields and random walks.
\newblock To appear in \textit{Ann. Inst. Henri Poincaré Probab. Stat.}
  arXiv:1908.06732, 2021.

\bibitem{MarcusRosen2006MarkovGaussianLocTime}
{\sc Marcus, M.~B., and Rosen, J.}
\newblock {\em Markov processes, {G}aussian processes and local times},
  vol.~100.
\newblock Cambridge University Press, 2006.

\bibitem{Mehta04RandomMatrices}
{\sc Mehta, M.~L.}
\newblock {\em Random Matrices}, 3rd~ed., vol.~142 of {\em Pure and Applied
  Mathematics}.
\newblock Academic Press, 2004.

\bibitem{MulaseWaldron03GSE}
{\sc Mulase, M., and Waldron, A.}
\newblock Duality of orthogonal and symplectic matrix integrals and
  quaternionic {F}eynman graphs.
\newblock {\em Communications in Mathematical Physics 240\/} (2003), 553--586.

\bibitem{RevuzYor1999BMGrundlehren}
{\sc Revuz, D., and Yor, M.}
\newblock {\em Continuous martingales and {B}rownian motion}, 3rd~ed., vol.~293
  of {\em Grundlehren der mathematischen Wissenschaften}.
\newblock Springer, 1999.

\bibitem{RogersShi93DysonBM}
{\sc Rogers, L., and Shi, Z.}
\newblock Interacting {B}rownian particles and the {W}igner law.
\newblock {\em Probability Theory and Related Fields 95\/} (1993), 555–570.

\bibitem{Symanzik65Scalar}
{\sc Symanzik, K.}
\newblock {\em Euclidean quantum field theory {I}: {E}quations for a scalar
  model}.
\newblock New York University, 1965.

\bibitem{Symanzik66Scalar}
{\sc Symanzik, K.}
\newblock Euclidean quantum field theory {I}. {E}quations for a scalar model.
\newblock {\em Journal of Mathematical Physics 7}, 3 (1966), 510--525.

\bibitem{Symanzik1969QFT}
{\sc Symanzik, K.}
\newblock Euclidean quantum field theory.
\newblock In {\em Scuola intenazionale di Fisica Enrico Fermi. XLV Corso.\/}
  (1969), Academic Press, pp.~152--223.

\bibitem{Sznitman2012LectureIso}
{\sc Sznitman, A.-S.}
\newblock {\em Topics in occupation times and {G}aussian free field}.
\newblock Zurich lectures in advanced mathematics. European Mathemtical
  Society, 2012.

\bibitem{Yor97AspectsBMII}
{\sc Yor, M.}
\newblock {\em Some aspects of {B}rownian motion. Part {II}: Some recent
  martingale problems}, 1st~ed.
\newblock Lectures in Mathematics ETH Zürich. Springer, 1997.

\end{thebibliography}

\end{document}